\documentclass[12pt,reqno]{amsart}

\usepackage[usenames, dvipsnames, svgnames]{xcolor}
\usepackage{amsmath, amssymb,graphicx,amsthm,latexsym, amsfonts, enumitem, mathtools, tensor}
\usepackage{hyperref}
\usepackage[all, color]{xy}
\usepackage{color}
\usepackage{amssymb}
\usepackage{float}
\usepackage{tikz}
\usetikzlibrary{arrows,decorations.pathmorphing,decorations.pathreplacing,positioning,shapes.geometric,shapes.misc,decorations.markings,decorations.fractals,calc,patterns}

\newcommand{\C}{\mathbb{C}}
\DeclareMathOperator{\Hom}{Hom}
\DeclareMathOperator{\End}{End}

\DeclareMathOperator{\Mod}{Mod}
\DeclareMathOperator{\mmod}{mod}

\DeclareMathOperator{\rad}{rad}
\DeclareMathOperator{\Ext}{Ext}

\theoremstyle{plain}
\newtheorem{theorem}{Theorem}[section]
\newtheorem*{theorem*}{Theorem}

\theoremstyle{definition}
\newtheorem{defn}[theorem]{Definition}

\newtheorem{exmp}[theorem]{Example} 
\newtheorem{remark}[theorem]{Remark}

\newtheorem{lemma}[theorem]{Lemma}

\newtheorem{corollary}[theorem]{Corollary}

\newtheorem{setup}[theorem]{Setup}
\newtheorem{proposition}[theorem]{Proposition}

\date{}
\begin{document}
\setlength{\parindent}{0pt}
\setlength{\parskip}{7pt}
\title[Almost split morphisms in subcategories of triangulated categories]{Almost split morphisms in subcategories of triangulated categories}
\author{Francesca Fedele}
\address{Dipartimento di Informatica - Settore di Matematica, ¨ Universit\`a degli Studi di
Verona, Strada le Grazie 15 - Ca' Vignal, I-37134 Verona, Italy}
\email{francesca.fedele@univr.it}
 
\begin{abstract}
For a suitable triangulated category $\mathcal{T}$ with a Serre functor $S$ and a full precovering subcategory $\mathcal{C}$ closed under summands and extensions, an indecomposable object $C$ in $\mathcal{C}$ is called Ext-projective if Ext$^1(C,\mathcal{C})=0$. Then there is no Auslander-Reiten triangle in $\mathcal{C}$ with end term $C$. In this paper, we show that if, for such an object $C$, there is a minimal right almost split morphism $\beta:B\rightarrow C$ in $\mathcal{C}$, then $C$ appears in something very similar to an Auslander-Reiten triangle in $\mathcal{C}$: an essentially unique triangle in $\mathcal{T}$ of the form
\begin{align*}
    \Delta= X\xrightarrow{\xi} B\xrightarrow{\beta} C\rightarrow \Sigma X,
\end{align*}
where $X$ is an indecomposable not in $\mathcal{C}$ and $\xi$ is a $\mathcal{C}$-envelope of $X$.
Moreover, under some extra assumptions, we show that removing $C$ from $\mathcal{C}$ and replacing it with $X$ produces a new subcategory of $\mathcal{T}$ closed under extensions. We prove that this process coincides with the classic mutation of $\mathcal{C}$ with respect to the rigid subcategory of $\mathcal{C}$ generated by all the indecomposable $\Ext$-projectives in $\mathcal{C}$ apart from $C$.

When $\mathcal{T}$ is the cluster category of Dynkin type $A_n$ and $\mathcal{C}$ has the above properties, we give a full description of the triangles in $\mathcal{T}$ of the form $\Delta$ and show under which circumstances replacing $C$ by $X$ gives a new extension closed subcategory.
\end{abstract}
\maketitle
\section{Introduction}
For $\Lambda$ a finite-dimensional algebra, consider the abelian category $\mmod \Lambda$ of finitely generated (right) $\Lambda$-modules. Auslander-Reiten sequences are a very useful tool to study $\mmod \Lambda$. These are non-splitting short exact sequences in $\mmod \Lambda$ of the form
\begin{align*}
    0\rightarrow A\xrightarrow{\alpha} B\xrightarrow{\beta} C\rightarrow 0,
\end{align*}
that are ``as close as possible'' to split exact sequences. Important properties of such a sequence are that $A,\, C$ are indecomposable modules determining each other and the morphisms $\alpha,\,\beta$ determine all the irreducible morphisms from $A$ and those ending at $C$.

The theory of Auslander-Reiten sequences has been extended to general abelian categories in \cite{AR} and to the study of Auslander-Reiten sequences in their subcategories, conducted by Auslander and Smal\o\,  in \cite{AS}.

Taking inspiration from the above, Happel developed the theory of Auslander-Reiten triangles in triangulated categories in \cite{DH} and then J\o rgensen defined and studied Auslander-Reiten triangles in their non-triangulated subcategories in \cite{JP}. Let $k$ be a field, $\mathcal{T}$ a skeletally small $k$-linear $\Hom$-finite triangulated category with split idempotents having a Serre functor $S$. Let $\mathcal{C}\subseteq \mathcal{T}$ be a full subcategory closed under summands and extensions. Any indecomposable object $C$ in $\mathcal{C}$ has an Auslander-Reiten triangle in $\mathcal{T}$ of the form
\begin{align*}
    \tau C\rightarrow Y\rightarrow C\rightarrow\Sigma X,
\end{align*}
where $\tau C=S\Sigma^{-1} C$, see \cite[proposition I.2.3 and its proof]{RV}. The main theorem in \cite{JP} shows that, if $\Hom(C,\Sigma \mathcal{C})$ is non-zero, then there is an Auslander-Reiten triangle in $\mathcal{C}$ of the form
\begin{align*}
    A\rightarrow B\rightarrow C\rightarrow\Sigma A
\end{align*}
if and only if there is a $\mathcal{C}$-cover $A\rightarrow \tau C$.

Auslander-Reiten theory has proved to be a useful tool in several different contexts and it has been studied in depth and generalised in various different ways, see for example \cite{JP2}, \cite{Liu}, \cite{Liu2} and \cite{R}.

Here we focus on the objects for which the main theorem in \cite{JP} cannot be applied, \textbf{i.e.} the objects $C$ in $\mathcal{C}$ with $\Hom(C,\Sigma\mathcal{C})=0$, which we call \textit{$\Ext$-projectives}. Similarly, the objects for which the dual of the above theorem cannot be applied are called $\Ext$-\textit{injectives}.
Some of the results we prove about these objects and the triangles they appear in are inspired by the ones on $\Ext$-projective (and $\Ext$-injective) modules and the properties of the short exact sequences they appear in, proven by Kleiner in \cite{MK}.

Note that an $\Ext$-projective object $C$ cannot appear in an Auslander-Reiten triangle in $\mathcal{C}$ of the form
\begin{align*}
    A\rightarrow B\rightarrow C \xrightarrow{\gamma} \Sigma A,
\end{align*}
since $\gamma=0$ would be forced, contradicting Definition \ref{defnARsub} below. However, as shown in the following theorem, for a suitable subcategory $\mathcal{C}$, we can find something quite similar to an Auslander-Reiten triangle in $\mathcal{C}$.

{\bf Theorem A (=Theorem \ref{prop_mras}).}
{\em
Let $\beta:B\rightarrow C$ be a minimal right almost split morphism in $\mathcal{C}$ with $C$ $\Ext$-projective.
\begin{enumerate}[label=(\alph*)]
\item The triangle
\begin{align*}
    \Delta: X\xrightarrow{\xi} B\xrightarrow{\beta} C \rightarrow \Sigma X
\end{align*}
is such that $X$ is an indecomposable object not in $\mathcal{C}$ and $\xi$ is a $\mathcal{C}$-envelope of $X$.
\item In part (a), the end terms $X$ and $C$ determine each other. That is, suppose $\beta':B'\rightarrow C'$ is another minimal right almost split morphism in $\mathcal{C}$ with $C'$ $\Ext$-projective and extend it to a triangle: $X'\rightarrow B'\xrightarrow{\beta'} C'\rightarrow \Sigma X'$. Then $C'\cong C$ if and only if $X'\cong X$.
\end{enumerate}
}

For their similarity with Auslander-Reiten triangles, we call the triangles of the form $\Delta$ \textit{left-weak Auslander-Reiten triangles in $\mathcal{C}$}. 
Note that $\beta:B\rightarrow C$, and hence $\Delta$, exist in fairly general circumstances, for example if $C$ is indecomposable and $\mathcal{C}$ is functorially finite in $\mathcal{T}$, see \cite[propositions\ 2.10 and 2.11]{IY}.

In \cite{IY}, Iyama and Yoshino defined the mutation of a subcategory of $\mathcal{T}$ with respect to a rigid subcategory $\mathcal{D}$ of $\mathcal{T}$. Under some assumptions, mutating an extension closed subcategory $\mathcal{C}$ of $\mathcal{T}$ with respect to a rigid $\mathcal{D}\subseteq\mathcal{C}$, gives a new extension closed subcategory of $\mathcal{T}$, see \cite[theorem 3.3]{ZZ}. We introduce a similar process to this and show how, in some cases, removing the third term of a left-weak Auslander-Reiten triangle $\Delta$ in $\mathcal{C}$ and replacing it with the first term of $\Delta$, gives a new extension closed subcategory of $\mathcal{T}$.
Let $\text{Ind}$ denote the indecomposable objects of a subcategory.

{\bf Theorem B (=Theorem \ref{thm_mutation}).}
{\em
Assume $\mathcal{C}$ is functorially finite in $\mathcal{T}$ and $C\in \mathcal{C}$ is an indecomposable $\Ext$-projective. Then there is a left-weak Auslander-Reiten triangle in $\mathcal{C}$ of the form
\begin{align}\label{triangle2}
X\xrightarrow{\xi} B\xrightarrow{\beta} C\xrightarrow{\gamma}\Sigma X. \tag{$\star$}
\end{align}
Let $\widetilde{\mathcal{C}}$ be the additive subcategory with $\text{Ind }\widetilde{\mathcal{C}}=\text{Ind }(\mathcal{C})\setminus C$ and define $\mathcal{C}':=\text{add }(\widetilde{\mathcal{C}}\cup X)$.
\begin{enumerate}[label=(\alph*)]
     \item If $X\in P(\mathcal{C}')\cap I(\mathcal{C}')$, then $\mathcal{C}'$ is closed under extensions.
    \item If $\End(X)$ and $\End(C)$ are division rings and $\mathcal{C}'$ is closed under extensions, then $X\in I(\mathcal{C}')$.
    \item If $\End(C)$ is a division ring, $\mathcal{C}'$ is closed under extensions and $X\in P(\mathcal{C}')$, then $X\in I(\mathcal{C}')$.
\end{enumerate}
}

Moreover, we show that in some cases this process and the classic mutation from \cite{IY} coincide.

{\bf Theorem C (=Theorem \ref{thm_mutZZ_same}).}
{\em
In the setup of Theorem B, suppose that $\mathcal{T}$ is $2$-Calabi-Yau, $\mathcal{C}$ has finitely many indecomposables and $X$ is $\Ext$-projective in $\mathcal{C}'$.
Let $\mathcal{D}$ be the additive subcategory generated by the $\Ext$-projectives in $\widetilde{\mathcal{C}}$ and $\mu (\mathcal{C}; \mathcal{D})$ be the classic (backward) $\mathcal{D}$-mutation of $\mathcal{C}$. Then, we have
\begin{align*}
    \mu (\mathcal{C}; \mathcal{D})=\mathcal{C}',
\end{align*}
 and this is a subcategory of $\mathcal{T}$ closed under extensions.
}

{\bf Remark.}
We apply our results to $\mathcal{C}_{A_n}$, the cluster category of Dynkin type $A_n$.
By \cite{HJR}, a subcategory $\mathcal{C}\subseteq \mathcal{C}_{A_n}$ is closed under extensions and direct summands if and only if it corresponds to a so-called Ptolemy diagram of the regular $(n+3)$-gon $P$. Moreover, we show that indecomposable $\Ext$-projectives in such a $\mathcal{C}$ are \textit{dissecting diagonals} in the corresponding Ptolemy diagram, {\bf i.e.} those diagonals dividing $P$ into cells.

We apply Theorem A to this example to give a complete description of the left-weak Auslander-Reiten triangles in $\mathcal{C}$. We show that Theorem B can be applied to an indecomposable $\Ext$-projective $C$ in $\mathcal{C}$ if and only if $C$ borders two empty cells in the Ptolemy diagram corresponding to $\mathcal{C}$.
Moreover, note that $\mathcal{C}_{A_n}$ is $2$-Calabi-Yau and it has finitely many indecomposables. Hence, whenever $C\in\mathcal{C}$ corresponds to a dissecting diagonal bordering two empty cells, Theorem C implies that $\mathcal{C}'$ is the subcategory obtained by mutating $\mathcal{C}$ with respect to the additive subcategory of $\mathcal{C}$ generated by all the indecomposable $\Ext$-projectives in $\widetilde{\mathcal{C}}$. 

The paper is organized as follows. Section \ref{section2} consists of the setup for $\mathcal{T}$ and $\mathcal{C}$ and some definitions. In Section \ref{section3} we present $\Ext$-projectives and prove Theorem A. In Section \ref{section4} we prove Theorem B. In Section \ref{section_mutation_same} we recall the classic mutation and prove Theorem C. Finally, Section \ref{section6} is an application of our results to $\mathcal{C}_{A_n}$.

\section{Setup and some definitions}\label{section2}
We work in the following setup, where additive subcategory means full subcategory closed under isomorphisms, sums and summands.
\begin{setup}\label{setup}
Let $k$ be a field, $\mathcal{T}$ be a skeletally small $k$-linear triangulated category with split idempotents in which each $\Hom$ space is finite dimensional over $k$. Note that this implies that $\mathcal{T}$ is a Krull-Schmidt category  \cite[remark \ 1.2]{JP}. Assume that $\mathcal{T}$ has a Serre functor $S$, see \cite[section I.1]{RV}. Also, let $\mathcal{C}$ be an additive subcategory of $\mathcal{T}$ closed under extensions.
\end{setup}
We present some terminology that will be used in the following sections. By $\Sigma$ we always denote the suspension in the category $\mathcal{T}$.

\begin{defn}
Let $A,B,C$ be in $\mathcal{C}$. We say that
\begin{enumerate} [label=(\alph*)]
\item a morphism $\alpha: A\rightarrow B$ is \textit{left almost split in $\mathcal{C}$} if it is not a split monomorphism and for every $A'$ in $ \mathcal{C}$, every morphism $\alpha':A\rightarrow A'$ which is not a split monomorphism factors through $\alpha$, \textbf{i.e.} there exists a morphism $B\rightarrow A'$ such that the following diagram commutes:
\begin{align*}
\xymatrix @C=1.5em @R=1em{
A\ar[rr]^\alpha \ar[dr]_{\alpha'}& & B\ar@{-->}[ld]^{\exists};\\
& A' &
}
\end{align*}
\item a morphism $\beta:B\rightarrow C$ is \textit{right almost split in $\mathcal{C}$} if it is not a split epimorphism and for every $C'$ in $ \mathcal{C}$, every morphism $\gamma: C'\rightarrow C$ which is not a split epimorphism factors through $\beta$, \textbf{i.e.} there exists a morphism $C'\rightarrow B$ such that the following diagram commutes:
\begin{align*}
\xymatrix @C=1.5em @R=1em{
B\ar[rr]^\beta & & C.\\
& C'\ar@{-->}[lu]^{\exists}\ar[ru]_\gamma &
}
\end{align*}
\end{enumerate}
\end{defn}

\begin{defn}[{\cite[definition\ 1.3]{JP}}]\label{defnARsub}
A distinguished triangle in $\mathcal{T}$ of the form
    \begin{align*}
A\xrightarrow{\alpha} B\xrightarrow{\beta} C\xrightarrow{\gamma} \Sigma A
\end{align*}
with $A,B,C\in \mathcal{C}$ is an \textit{Auslander-Reiten triangle in $\mathcal{C}$} if the following are satisfied:
\begin{enumerate}
\item the morphism $\gamma$ is non-zero,
\item the morphism $\alpha$ is left almost split in $\mathcal{C}$,
\item the morphism $\beta$ is right almost split in $\mathcal{C}$.
\end{enumerate}
\end{defn}

\begin{remark}
Note that in the above definition, condition (1) is implied by each of the other two conditions.
\end{remark}

\begin{defn}
A morphism $\xi:X\rightarrow Y$ is \textit{right minimal} if $\xi\circ \varphi =\xi$, for   $\varphi: X\rightarrow X$, implies that $\varphi$ is an automorphism. \textit{Left minimal} morphisms are defined dually.
\end{defn}

\begin{defn}
A morphism in $\mathcal{C}$ is called \textit{minimal right almost split in $\mathcal{C}$} if it is both right minimal and right almost split in $\mathcal{C}$. Similarly, a morphism in $\mathcal{C}$ is \textit{minimal left almost split in $\mathcal{C}$} if it is both left minimal and left almost split in $\mathcal{C}$. 
\end{defn}

The following is well known for the module category case, see \cite[proposition V.1.4 and lemma V.1.7]{ARS}. We prove here the corresponding result for our setup.
\begin{lemma}\label{lemma_ras_ind}
\begin{enumerate}[label=(\alph*)]
\item Let $\beta:B\rightarrow C$ be right almost split in $\mathcal{C}$, then $C$ is indecomposable. Moreover, if $\beta$ is also right minimal and $\beta': B'\rightarrow C$ is another minimal right almost split morphism in $\mathcal{C}$, then there is an isomorphism $\varphi:B\rightarrow B'$ such that $\beta'\circ \varphi=\beta$.
\item Let $\alpha:A\rightarrow B$ be left almost split in $\mathcal{C}$, then $A$ is indecomposable.  Moreover, if $\alpha$ is also left minimal and $\alpha': A\rightarrow B'$ is another minimal left almost split morphism in $\mathcal{C}$, then there is an isomorphism $\varphi:B\rightarrow B'$ such that $\varphi\circ \alpha=\alpha'$.
\end{enumerate}
\end{lemma}

\begin{proof}
(a) In order to prove that $C$ is indecomposable, it is enough to prove that $\End C$,  the endomorphism ring of $C$, is local.
Let $\gamma_0, \gamma_1:C\rightarrow C$ be elements in $\End C$ without right inverses, \textbf{i.e.} $\gamma_0,\gamma_1$ are not split epimorphisms.
Then there are $\gamma_0 ', \gamma_1 ' :C\rightarrow B$ such that $\gamma_i =\beta \circ \gamma_i '$ for $i=0,1$ and we get $\gamma_0 + \gamma_1= \beta\circ (\gamma_0 '+\gamma_1 ')$. If $\gamma_0 + \gamma_1$ had a right inverse $\delta$, then 
\begin{align*}
1_C=(\gamma_0+\gamma_1)\circ \delta=\beta\circ (\gamma_0 '+ \gamma_1')\circ \delta,
\end{align*}
so that $\beta$ would also have a right inverse, this is a contradiction. Hence the set of elements of $\End C$ without right inverses is closed under addition and so $\End C$ is local by  \cite[proposition\ 15.15]{AF}.

Assume now that $\beta$ is also right minimal and $\beta': B'\rightarrow C$ is another minimal right almost split morphism in $\mathcal{C}$. Then there are morphisms $\varphi: B\rightarrow B',\, \phi:B'\rightarrow B$ such that $\beta'\circ \varphi=\beta$ and $\beta\circ\phi=\beta'$. Then $\beta=\beta\circ \phi\circ \varphi$ and $\beta'=\beta'\circ \varphi\circ\phi$. By right minimality of $\beta$ and $\beta'$, it follows that $\varphi\circ \phi$ and $\phi\circ\varphi$ are isomorphisms and so $\varphi$ is an isomorphism.

(b) This follows by a similar argument.
\end{proof}

\begin{defn}
[{\cite[definition \ 1.4]{JP}}]
Let $X\in\mathcal{T}$. A $\mathcal{C}$\textit{-precover} of $X$ is a morphism of the form $\alpha_X :A_X\rightarrow X$ with $A_X\in \mathcal{C}$ such that every morphism $\alpha':A'\rightarrow X$ with $A'\in \mathcal{C}$ factorizes as:
\begin{align*}
\xymatrix{
A'\ar[rr]^{\alpha'} \ar@{-->}[dr]_{\exists}& & X.\\
& A_X \ar[ru]_{\alpha_X} &
}
\end{align*}
A $\mathcal{C}$\textit{-cover} of $X$ is a $\mathcal{C}$-precover of $X$ which is also a right minimal morphism.
Note that given a $\mathcal{C}$-precover, we can obtain a $\mathcal{C}$-cover by dropping superfluous direct summands, see \cite[section 1]{IY}.
The dual notions of precovers and covers are \textit{preenvelopes} and \textit{envelopes} respectively.
\end{defn}

\begin{defn}
The subcategory $\mathcal{C}$ of $\mathcal{T}$ is called \textit{precovering} if every object in $\mathcal{T}$ has a $\mathcal{C}$-precover. Dually, $\mathcal{C}$ is \textit{preenveloping} if every object in $\mathcal{T}$ has a $\mathcal{C}$-preenvelope. If $\mathcal{C}$ is both precovering and preenveloping, we say that it is \textit{functorially finite}.
\end{defn}

\begin{defn}[{\cite[section I.2]{RV}}]\label{defn_AR_exist}
Since $\mathcal{T}$ has a Serre functor $S$, we can consider the functor $\tau:=S\Sigma^{-1}:\mathcal{T}\rightarrow\mathcal{T}$.  The functor $\tau$ is called \textit{Auslander-Reiten translation} and it is invertible with $\tau^{-1}=S^{-1}\Sigma$. Moreover, if $X$ is an indecomposable in $\mathcal{T}$, then there are Auslander-Reiten triangles in $\mathcal{T}$ of the form
\begin{align*}
X\rightarrow Y\rightarrow \tau^{-1}X\rightarrow \Sigma X \text{ and }
\tau X\rightarrow Z\rightarrow X\rightarrow \Sigma (\tau X).
\end{align*}
\end{defn}

\section{Ext-projectives and weak Auslander-Reiten triangles in $\mathcal{C}$}\label{section3}
In this section we introduce $\Ext$-projective (respectively $\Ext$-injective) objects in $\mathcal{C}$. We study the properties of the triangles they appear in, that we will call left-weak (respectively right-weak) Auslander-Reiten triangles in $\mathcal{C}$.
\begin{remark}
When $\mathcal{T}=D(\Lambda)$ is the derived category of some algebra $\Lambda$, we know that for every $X,\,Y\in\Mod \Lambda$ we have $\Hom_{D(\Lambda)}(X,\Sigma Y)\cong \Ext^1 _\Lambda (X,Y)$,  \cite[section\ I.6]{H}. In the general case, for $\mathcal{T}$ satisfying our setup, we define $\Ext^1(X,Y):=\Hom_{\mathcal{T}}(X,\Sigma Y)$ for all $X,\,Y\in \mathcal{T}$.
\end{remark}

\begin{defn}
An object $C\in \mathcal{C}$ is called $\Ext$\textit{-injective} if $\Ext^1(A,C)=0$ for all $A\in\mathcal{C}$. An object $D\in\mathcal{C}$ is called $\Ext$\textit{-projective} if $\Ext^1(D,A)=0$ for all $A\in\mathcal{C}$.
\end{defn}

\begin{lemma}\label{lemma_proj1}
\begin{enumerate}[label=(\alph*)]
\item Let $C\in \mathcal{C}$ be an indecomposable $\Ext$-projective object. For any non-split triangle $X\xrightarrow{\xi} B\xrightarrow{\beta} C\xrightarrow{\gamma\neq 0} \Sigma X$ with $B\in \mathcal{C}$, the morphism $\xi$ is a $\mathcal{C}$-envelope of $X$. If $\beta$ is a right minimal morphism, then $X$ is indecomposable.
\item Let $A\in\mathcal{C}$ be an indecomposable $\Ext$-injective object. For any non-split triangle $A\xrightarrow{\alpha} B\xrightarrow{\beta} Z\xrightarrow{\zeta\neq 0}\Sigma A$ with $B\in \mathcal{C}$, the morphism $\beta$ is a $\mathcal{C}$-cover of $Z$. If $\alpha$ is a left minimal morphism, then $Z$ is indecomposable.
\end{enumerate}
\end{lemma}

We present a lemma that we will use to prove Lemma \ref{lemma_proj1}.

\begin{lemma}\label{lemma_rm}
Suppose $\xi=(\xi', \xi''): X'\oplus X''\rightarrow Y$ is a right minimal morphism in $\mathcal{T}$. Then $\xi':X'\rightarrow Y$ is right minimal.
\end{lemma}

\begin{proof}
Consider $\varphi':X'\rightarrow X'$ such that $\xi'\circ \varphi' =\xi'$. Taking
\begin{align*}
    \varphi:=\begin{pmatrix} \varphi' & 0\\ 0& 1_{X''} \end{pmatrix}:X'\oplus X''\rightarrow X'\oplus X'',
\end{align*}
we have
\begin{align*}
    \xi\circ \varphi=(\xi'\circ \varphi', \xi''\circ 1_{X''})=(\xi',\xi'')=\xi.
\end{align*}
As $\xi$ is right minimal, then $\varphi$ is an isomorphism and hence $\varphi'$ is also an isomorphism, meaning that $\xi'$ is right minimal.
\end{proof}

\begin{proof}[Proof of Lemma \ref{lemma_proj1}]
(a) 
Let $D\in\mathcal{C}$ and apply $\Hom_{\mathcal{T}}(-, D)$ to the triangle $\Sigma^{-1} C\xrightarrow{-\Sigma^{-1} \gamma} X\xrightarrow{\xi} B\xrightarrow{\beta}C$ to obtain the exact sequence:
\begin{align*}
    \Hom_{\mathcal{T}}(B, D)\xrightarrow{\Hom_{\mathcal{T}}(\xi, D)} \Hom_{\mathcal{T}}(X, D)\rightarrow \Hom_{\mathcal{T}}(\Sigma^{-1} C, D).
\end{align*}
Since $C$ is $\Ext$-projective in $\mathcal{C}$, then $\Hom_{\mathcal{T}}(C, \Sigma D)=0$ and hence $\Hom_{\mathcal{T}}(\Sigma^{-1} C, D)=0$. Then $\Hom_{\mathcal{T}}(\xi, D)$ is surjective and so every $\eta\in \Hom_{\mathcal{T}}(X, D)$ factors as $\eta=\epsilon\circ \xi$ for some $\epsilon\in \Hom_{\mathcal{T}}(B, D)$. Since this is true for every $D\in \mathcal{C}$, it follows that $\xi$ is a $\mathcal{C}$-preenvelope of $X$.

Since $C$ is indecomposable, we can write $\beta=(\beta_1,\dots,\,\beta_t):B=B_1\oplus\dots\oplus B_t\rightarrow C$, where $B_1,\dots,\, B_t$ are indecomposable. Then, by \cite[proposition 3.5, appendix]{A}, each $\beta_i$ is either an isomorphism or it is in $\rad_\mathcal{T}$. Since the triangle extending $\xi$ is not split, then $\beta$ is not a split epimorphism and each $\beta_i$ is in the radical. Hence $\beta$ is in the radical. This implies that $\xi$ is left minimal and hence it is a $\mathcal{C}$-envelope of $X$.

Suppose now that $\beta$ is a right minimal morphism and let $X=X_1\oplus \dots\oplus X_r$ be the indecomposable decomposition of $X$. Note that $\xi$ is the direct sum of $\mathcal{C}$-envelopes $\xi_i :X_i\rightarrow B_i$, for $i=1,\dots, r$. In fact, by \cite[section\ 1]{MK}, $\mathcal{C}$-envelopes are unique up to isomorphism, so the direct sum of $\mathcal{C}$-envelopes of $X_i$'s has to be isomorphic to $\xi$. Then, for each $i=1,\dots, r$ we have commutative diagram
\begin{align*}
    \xymatrix{
    X_i\ar[r]^{\xi_i}\ar[d]^{\iota_i} & B_i \ar[d]^{\overline{\iota_i}}\\
    X\ar[r]^{\xi} & B,
    }
\end{align*}
where $\iota_i,\, \overline{\iota_i}$ are the inclusions. Completing $\xi_i$ to a triangle, we get a commutative diagram
\begin{align}\label{fig:triangles_i}
    \xymatrix{
    X_i\ar[r]^{\xi_i}\ar[d]^{\iota_i} & B_i\ar[r]^{\delta_i}\ar[d]^{\overline{\iota_i}} & Z_i \ar[r]^{\zeta_i}\ar[d]^{\eta_i}& \Sigma X_i\ar[d]^{\Sigma \iota_i}\\
    X\ar[r]^\xi &B\ar[r]^\beta & C\ar[r]^\gamma &\Sigma X,
    }
\end{align}
where $\eta_i$ exists by the axioms of triangulated categories. Note that the direct sum for $i=1,\dots, r$ of triangles of the first row of (\ref{fig:triangles_i}) is isomorphic to the triangle in the second row. In particular, $C\cong Z_1\oplus \dots\oplus Z_r$, but since $C$ is indecomposable, without loss of generality we have $C\cong Z_1$ and $Z_i=0$ for $i\neq 1$.

Note that since $\beta$ is right minimal, by Lemma \ref{lemma_rm} so is its restriction to $B_i$, say $\beta_i:B_i\rightarrow C$.
For $i\neq 1$, we then have $\beta_i=\beta\circ\overline{\iota_i}=\eta_i\circ\delta_i=0$ right minimal and so $B_i=0$. But then, as $\xi_i$ is an isomorphism, it follows that $X_i=0$ and hence $X\cong X_1$ is indecomposable.

(b) This is proven in a similar way.
\end{proof}

\begin{theorem}\label{prop_mras}
Let $\beta:B\rightarrow C$ be a minimal right almost split morphism in $\mathcal{C}$ with $C$ $\Ext$-projective.
\begin{enumerate}[label=(\alph*)]
\item The triangle
\begin{align}\label{triangle}
X\xrightarrow{\xi} B\xrightarrow{\beta} C\xrightarrow{\gamma}\Sigma X \tag{$\star$}
\end{align}
is such that $X$ is an indecomposable not in $\mathcal{C}$ and $\xi$ is a $\mathcal{C}$-envelope of $X$.
\item If $\beta': B'\rightarrow C'$ is a minimal right almost split morphism in $\mathcal{C}$ with $C'$ $\Ext$-projective, then $C'\cong C$ if and only if $X'\cong X$, where $X'\xrightarrow{\xi'}B'\xrightarrow{\beta'}C'\xrightarrow{\gamma'}\Sigma X'$ is the triangle obtained by extending $\beta'$.
\end{enumerate}
\end{theorem}

\begin{proof}
(a) Note that since $\beta$ is right almost split in $\mathcal{C}$, then $C$ is indecomposable by Lemma \ref{lemma_ras_ind}.
If $X$ was in $\mathcal{C}$, as $C$ is $\Ext$-projective, we would have $\Ext^{1}(C,X)=\Hom (C,\Sigma X)=0$ and hence $\gamma=0$ and the triangle (\ref{triangle}) splitting, contradicting $\beta$ being right almost split. Hence $X\not\in \mathcal{C}$ and $\gamma\neq 0$. Then, by Lemma \ref{lemma_proj1}, it follows that $\xi$ is a $\mathcal{C}$-envelope of $X$ and, since $\beta$ is right minimal, $X$ is indecomposable.

(b) Assume now that $\beta':B'\rightarrow C'$ is a minimal right almost split morphism in $\mathcal{C}$ with $C'$ $\Ext$-projective, and extend it to a triangle:
\begin{align*}
X'\xrightarrow{\xi'} B'\xrightarrow{\beta'} C'\xrightarrow{\gamma'}\Sigma X'.
\end{align*}
By the argument above, $X'\not\in\mathcal{C}$ is indecomposable and $\xi'$ is a $\mathcal{C}$-envelope of $X'$.

Suppose first that $C'\cong C$, say that $\varphi: C\rightarrow C'$ is an isomorphism. Since $\beta'$ and $\varphi\circ \beta$ are minimal right almost split morphisms with codomain $C'$, it follows from Lemma \ref{lemma_ras_ind} that there is an isomorphism $\psi: B\rightarrow B'$ with $\beta'\circ\psi=\varphi\circ \beta$.
By the axioms of triangulated categories, there is a morphism $\rho:X\rightarrow X'$ making the following diagram commutative:
\begin{align*}
    \xymatrix{
    X\ar[r]^\xi \ar[d]^\rho & B\ar[r]^\beta \ar[d]^\psi & C\ar[r]^\gamma \ar[d]^\varphi &\Sigma X\ar[d]^{\Sigma\rho} \\
    X'\ar[r]^{\xi' }&B'\ar[r]^{\beta'}& C'\ar[r]^{\gamma'} &\Sigma X'.
    }
\end{align*}
By the $5$-Lemma \cite[exercise\ 10.2.2]{W}, it follows that $\rho$ is an isomorphism, so that $X\cong X'$.

Suppose now that $X\cong X'$, say that $\rho:X\rightarrow X'$ is an isomorphism. Since a $\mathcal{C}$-envelope of $X$ is unique up to isomorphism and $\xi$, $\xi'\circ\rho$ both are $\mathcal{C}$-envelopes of $X$, there exists an isomorphism $\psi:B\rightarrow B'$ such that $\psi\circ \xi=\xi'\circ \rho$. Then, by the axioms of triangulated categories and the $5$-Lemma, there is an isomorphism $\varphi$ between $C$ and $C'$.
\end{proof}

We state, without proof, the dual of Theorem \ref{prop_mras}.
\begin{theorem}\label{prop_mras_dual}
Let $\alpha:A\rightarrow B$ be a minimal left almost split morphism in $\mathcal{C}$ with $A$ $\Ext$-injective.
\begin{enumerate}[label=(\alph*)]
\item The triangle
\begin{align*}
    A\xrightarrow{\alpha} B\xrightarrow{\beta} Z\xrightarrow{\zeta}\Sigma A 
\end{align*}
is such that $Z$ is an indecomposable not in $\mathcal{C}$ and $\beta$ is a $\mathcal{C}$-cover of $Z$. 
\item If $\alpha':A'\rightarrow B'$ is a minimal left almost split morphism in $\mathcal{C}$ with $A'$ $\Ext$-injective, then $A'\cong A$ if and only if $Z'\cong Z$, where $A'\xrightarrow{\alpha'}B'\xrightarrow{\beta'}Z'\xrightarrow{\zeta'}\Sigma A'$ is the triangle obtained by extending $\alpha'$.
\end{enumerate}
\end{theorem}

Note that, even though the second morphism in the triangles from Theorem \ref{prop_mras} is minimal right almost split in $\mathcal{C}$, these are not Auslander-Reiten triangles in $\mathcal{C}$ since the first object in them is not in $\mathcal{C}$. Because of this ``weakness'' they have, we define them as follows.

\begin{defn}
A triangle in $\mathcal{T}$ of the form
\begin{align*}
X\xrightarrow{\xi} B\xrightarrow{\beta} C\xrightarrow{\gamma}\Sigma X 
\end{align*}
is called a \textit{left-weak Auslander-Reiten triangle in $\mathcal{C}$} if $C$ is Ext-projective in $\mathcal{C}$, $\beta$ is a minimal right almost split morphism in $\mathcal{C}$ and $\xi$ is a $\mathcal{C}$-envelope of $X$. Note that, by Theorem \ref{prop_mras}, the two end terms of the triangle are indecomposable objects determining each other. 

Dually, a triangle in $\mathcal{T}$ of the form
\begin{align*}
    A\xrightarrow{\alpha} B\xrightarrow{\beta} Z\xrightarrow{\zeta}\Sigma A 
\end{align*}
is called a \textit{right-weak Auslander-Reiten triangle in $\mathcal{C}$} if $A$ is Ext-injective in $\mathcal{C}$, $\alpha$ is a minimal left almost split morphism in $\mathcal{C}$ and $\beta$ is a $\mathcal{C}$-cover of $Z$. Note that, by Theorem \ref{prop_mras_dual}, the two end terms of the triangle are indecomposable objects determining each other. 
%
\end{defn}

\begin{remark}\label{rmk_star_exists}
Suppose $\mathcal{C}$ is functorially finite in $\mathcal{T}$. Then, by \cite[propositions 2.10 and 2.11]{IY} for any indecomposable object $C$ in $\mathcal{C}$ there is a minimal right almost split morphism in $\mathcal{C}$ ending at it and a minimal left almost split morphism in $\mathcal{C}$ starting at it. Hence, by Theorems \ref{prop_mras} and \ref{prop_mras_dual}, there is a left-weak Auslander-Reiten triangle in $\mathcal{C}$ ending at $C$ and a right-weak Auslander-Reiten triangle in $\mathcal{C}$ starting at $\mathcal{C}$.
\end{remark}

We end this section by giving equivalent definitions to $\Ext$-projectivity in the case when $\mathcal{C}$ is precovering in $\mathcal{T}$.

\begin{proposition}\label{thm_extproj_iff}
Assume $\mathcal{C}$ is precovering in $\mathcal{T}$. Let $C\in\mathcal{C}$ be indecomposable and $\alpha:A\rightarrow \tau C$ be a $\mathcal{C}$-cover. Then, the following are equivalent:
\begin{enumerate}[label=(\alph*)]
    \item $C$ is $\Ext$-projective in $\mathcal{C}$,
    \item $A=0$,
    \item $\Ext^1(C, A)=0$.
\end{enumerate}
\end{proposition}

\begin{proof}
Note that (a) implies (c) by definition of $\Ext$-projectivity, since $A\in\mathcal{C}$. The fact that (b) implies (c) is also clear.

To prove that (c) implies (a), assume that $C$ is not $\Ext$-projective. By \cite[theorem\ 3.1]{JP}, since $\alpha :A\rightarrow \tau C$ is a $\mathcal{C}$-cover, there is an Auslander-Reiten triangle in $\mathcal{C}$ of the form
\begin{align*}
    A\rightarrow B\rightarrow C\xrightarrow{\neq 0}\Sigma A.
\end{align*}
Hence $\Ext^1(C,A)=\Hom(C,\Sigma A)\neq 0$.

To prove that (c) implies (b), note that $C=\tau^{-1}(\tau C)$. Then, letting $D(-)=\Hom_k(-,k)$, we have
\begin{align*}
    0&=\Ext^1(C, A)=\Hom(C,\Sigma A)=\Hom(\tau^{-1}(\tau C),\Sigma A)\cong \Hom (\tau C, \tau\Sigma A)\\&\cong \Hom (\tau C, S A)\cong D\Hom (A,\tau C).
\end{align*}
Then, $\Hom(A,\tau C)=0$ and in particular $\alpha=0$. Since $\alpha$ is right minimal, it follows that $A=0$.
\end{proof}

The dual of the above follows in a similar way. Here we state it without proof.

\begin{proposition}
Assume $\mathcal{C}$ is preenveloping in $\mathcal{T}$. Let $A\in \mathcal{C}$ be indecomposable and $\zeta:\tau^{-1}A\rightarrow C$ be a $\mathcal{C}$-envelope. Then, the following are equivalent:
\begin{enumerate}[label=(\alph*)]
    \item $A$ is $\Ext$-injective in $\mathcal{C}$,
    \item $C=0$,
    \item $\Ext^1(C, A)=0$.
\end{enumerate}
\end{proposition}

\section{Extension closed subcategories from weak Auslander-Reiten triangles in $\mathcal{C}$}\label{section4}
 In this section, we show how, in some cases, it is possible to construct a new extension closed subcategory $\mathcal{C'}\subseteq \mathcal{T}$ modifying $\mathcal{C}$ using the objects that appear in a left-weak (or a right-weak) Auslander-Reiten triangle in $\mathcal{C}$. The idea of how this is done is similar to the mutation from \cite{IY}. 

\begin{defn}
For an additive subcategory $\mathcal{X}$ of $\mathcal{T}$, we denote by $\text{Ind }\mathcal{X}$ a maximal set of pairwise non-isomorphic indecomposable objects in $\mathcal{X}$.
\end{defn}

One of the conditions in the following lemma is that the endomorphism ring of the indecomposable $C$ in $\mathcal{C}$ is a division ring. Note that, for example, this is true if $\End(C)\cong k$.

\begin{lemma}\label{lemma_C_preenv}
Assume $\mathcal{C}$ is functorially finite in $\mathcal{T}$. Let $C\in\mathcal{C}$ be an indecomposable object such that $\End(C)$ is a division ring. Let $X$ in $\mathcal{T}$ be an indecomposable object. Let $\widetilde{\mathcal{C}}$ be the additive subcategory of $\mathcal{T}$ with $\text{Ind }\widetilde{\mathcal{C}}=\text{Ind }(\mathcal{C})\setminus C$ and set $\mathcal{C}':=\text{add }(\widetilde{\mathcal{C}}\cup X)$.
Then $\mathcal{C}'$ is functorially finite in $\mathcal{T}$.
\end{lemma}

\begin{proof}
We show that $\mathcal{C}'$ is preenveloping in $\mathcal{T}$, the proof for $\mathcal{C}'\subseteq \mathcal{T}$ precovering is then dual.
By \cite[propositions 2.10, 2.11]{IY}, there exists a left almost split morphism in $\mathcal{C}$ of the form $\eta:C\rightarrow D$. Since $\End(C)$ is a division ring, then every endomorphism of $C$ is either an isomorphism or it is the zero morphism. Hence $C$ is not a direct summand of $D$ and $D\in\widetilde{C}:=\text{add }(\mathcal{C}\setminus C)$. For any $Z$ in $\mathcal{T}$, consider a $\mathcal{C}$-preenvelope $\zeta: Z\rightarrow C^n\oplus \widetilde{C}$, for some non-negative integer $n$ and $\widetilde{C}\in\widetilde{\mathcal{C}}$. Consider
\begin{align*}
    Z\xrightarrow{\zeta} C^n\oplus \widetilde{C}\xrightarrow{\begin{pmatrix}G&0\\0& 1_{\widetilde{C}}\end{pmatrix}} D^n\oplus \widetilde{C},
\end{align*}
where $G$ is the $n\times n$ matrix having $\eta$ in the diagonal entries and zero elsewhere. Any morphism $\varphi:Z\rightarrow \widetilde{C}'$, where $\widetilde{C}'\in\widetilde{\mathcal{C}}$, factors through $\zeta$. So there exists a morphism $\gamma=(\gamma_1,\dots,\gamma_n,\widetilde{\gamma}):C^n\oplus\widetilde{C}\rightarrow\widetilde{C}'$ such that $\gamma\circ\zeta=\varphi$. Note that, since $C$ is not a direct summand of $\widetilde{C}'$, then $\gamma_i$ is not a split monomorphism for $i=1,\dots,n$. Hence, since $\eta$ is left almost split in $\mathcal{C}$, there exists $\delta_i: D\rightarrow \widetilde{C}'$ such that $\delta_i\circ \eta=\gamma_i$. Let $\delta=(\delta_1,\dots, \delta_n, \widetilde{\gamma}):D^n\oplus\widetilde{C}\rightarrow \widetilde{C}'$. Then
\begin{align*}
    \varphi=\gamma\circ\zeta=\delta\circ\zeta', \text{ for } \zeta':=\begin{pmatrix}G&0\\0& 1_{\widetilde{C}}\end{pmatrix}\zeta.
\end{align*}
Hence $\zeta':Z\rightarrow D^n\oplus \widetilde{C}$ is a $\widetilde{\mathcal{C}}$-preenvelope.
Adding some copies of $X$ to $D^n\oplus\widetilde{C}$ if necessary, we then obtain a $\mathcal{C}'$-preenvelope of $Z$.
\end{proof}

\begin{defn}
Let $\mathcal{X}\subset\mathcal{T}$ be an additive subcategory. The additive subcategory of $\mathcal{X}$ consisting of all the $\Ext$-injective (respectively $\Ext$-projective) objects in $\mathcal{X}$ is denoted $I(\mathcal{X})$ (respectively $P(\mathcal{X})$).
\end{defn}

\begin{theorem}\label{thm_mutation}
Assume $\mathcal{C}$ is functorially finite in $\mathcal{T}$ and $C\in P(\mathcal{C})$ is an indecomposable. Then there is a left-weak Auslander-Reiten triangle in $\mathcal{C}$ of the form
\begin{align}
X\xrightarrow{\xi} B\xrightarrow{\beta} C\xrightarrow{\gamma}\Sigma X. \tag{$\star$}
\end{align}
Let $\widetilde{\mathcal{C}}\subseteq \mathcal{T}$ be the additive subcategory with $\text{Ind }\widetilde{\mathcal{C}}=\text{Ind }(\mathcal{C})\setminus C$ and define $\mathcal{C}':=\text{add }(\widetilde{\mathcal{C}}\cup X)$.
\begin{enumerate}[label=(\alph*)]
    \item If $X\in P(\mathcal{C}')\cap I(\mathcal{C}')$, then $\mathcal{C}'$ is closed under extensions.
    \item If $\End(X)$ and $\End(C)$ are division rings and $\mathcal{C}'$ is closed under extensions, then $X\in I(\mathcal{C}')$.
    \item If $\End(C)$ is a division ring, $\mathcal{C}'$ is closed under extensions and $X\in P(\mathcal{C}')$, then $X\in I(\mathcal{C}')$.
\end{enumerate}
\end{theorem}

\begin{proof}
First note that $(\star)$ exists by Remark \ref{rmk_star_exists}.

(a) Suppose $X$ is $\Ext$-injective and $\Ext$-projective in $\mathcal{C}'$.
Consider first a triangle with end terms $\widetilde{C''},\, \widetilde{C'}$ in $\widetilde{\mathcal{C}}$:
\begin{align*}
   \widetilde{C''}\rightarrow A \rightarrow \widetilde{C'}\rightarrow \Sigma \widetilde{C''}.
\end{align*}
Since $\mathcal{C}$ is closed under extensions and $\widetilde{\mathcal{C}}\subset \mathcal{C}$, then $A\in\mathcal{C}$. We prove that $C$ is not a direct summand of $A$, so that $A\in\widetilde{\mathcal{C}}\subset\mathcal{C}'$. Suppose for a contradiction that $A\cong \overline{A}\oplus C$ for some $\overline{A}\in\mathcal{C}$.
Note that any morphism $\widetilde{C}\rightarrow C$ with $\widetilde{C}\in\widetilde{\mathcal{C}}\subset\mathcal{C}$ is not a split epimorphism, so it factors  through $\beta$ since $\beta$ is right almost split in $\mathcal{C}$. Hence, by the axioms of triangulated categories, we obtain a morphism of triangles of the form:

\begin{align*}
    \xymatrix{
    \widetilde{C''}\ar[r]\ar[d]& \overline{A}\oplus C\ar[r]^-{\alpha}\ar[d]^{(0,1)}& \widetilde{C'}\ar[r]\ar[d]^-{\delta} &\Sigma \widetilde{C''}\ar[d]\\
    B\ar[r]_\beta& C\ar[r]_-{\gamma}&\Sigma X\ar[r]_{-\Sigma \xi}&\Sigma B.
    }
\end{align*}
Since $X$ is $\Ext$-injective in $\mathcal{C}'$ and $\widetilde{C'}\in\mathcal{C}'$, then $\delta=0$. Hence
\begin{align*}
    0=\delta\alpha=\gamma (0,1)=(0,\gamma),
\end{align*}
contradicting the fact that $\gamma$ is non-zero. So $A\in\widetilde{\mathcal{C}}\subset\mathcal{C}'$.

Consider now a triangle with end terms in $\mathcal{C}'$, say $\epsilon: C''\rightarrow A\rightarrow C'\rightarrow \Sigma C''$.  Then
\begin{align*}
\epsilon: X^t\oplus \widetilde{C''}\rightarrow A\rightarrow X^s\oplus \widetilde{C'}\xrightarrow{\gamma'} \Sigma X^t\oplus\Sigma \widetilde{C''},
\end{align*}
for some non-negative integers $s,\,t$ and some $\widetilde{C''},\,\widetilde{C'}\in\widetilde{\mathcal{C}}$. Note that since $X$ is $\Ext$-injective and $\Ext$-projective in $\mathcal{C}'$, we have
\begin{align*}
    \gamma'=\begin{psmallmatrix}0&0\\0&\overline{\gamma'}\end{psmallmatrix}: X^s\oplus\widetilde{C'}\rightarrow  \Sigma X^t\oplus\Sigma \widetilde{C''}.
\end{align*}
Hence $\epsilon$ is the direct sum of triangles of the form
\begin{align*}
    X^t\xrightarrow{1} X^t\rightarrow 0\rightarrow \Sigma X^t, \,\,\, 0\rightarrow X^s\xrightarrow{1}X^s\rightarrow 0\,\, \text{ and }\,\, \widetilde{C''}\rightarrow \overline{A}\rightarrow \widetilde{C'}\xrightarrow{\overline{\gamma'}}\Sigma \widetilde{C''}.
\end{align*}
Note that, as $\widetilde{C''},\,\widetilde{C'}\in\widetilde{\mathcal{C}}$, then $\overline{A}\in\widetilde{\mathcal{C}}$ and so $A\in\mathcal{C}'$. Hence $\mathcal{C}'$ is closed under extensions.

(b) Suppose now that $\End(X)$ and $\End(C)$ are division rings and $\mathcal{C}'$ is closed under extensions. Suppose for a contradiction that $X$ is not $\Ext$-injective in $\mathcal{C}'$. By Definition \ref{defn_AR_exist}, there is an Auslander-Reiten triangle in $\mathcal{T}$ of the form:
\begin{align*}
    X\rightarrow Y\rightarrow \tau^{-1}X\rightarrow \Sigma X.
\end{align*}
Also, since $\mathcal{C}\subseteq \mathcal{T}$ is functorially finite, then $\mathcal{C}'\subseteq\mathcal{T}$ is preenveloping by Lemma \ref{lemma_C_preenv}. Let $\tau^{-1}X\rightarrow D$ be a $\mathcal{C}'$-envelope. Then, by \cite[theorem 3.2]{JP}, there is an Auslander-Reiten triangle in $\mathcal{C}'$ of the form:
\begin{align*}
    X\xrightarrow{\xi'} E\xrightarrow{\epsilon} D\xrightarrow{\delta} \Sigma X.
\end{align*}
Since $\End(X)$ is a division ring, then every endomorphism of $X$ is either an isomorphism or it is the zero morphism. Then, since $\xi'$ is not a split monomorphism, we have that $X$ is not a direct summand of $E$ and so $E\in\widetilde{\mathcal{C}}\subset\mathcal{C}$. As $\xi:X\rightarrow B$ is a $\mathcal{C}$-envelope, then there exists a morphism $\varphi: B\rightarrow E$ such that $\varphi\circ \xi=\xi'$. Then, by the axioms of triangulated categories, we obtain a morphism of triangles of the form:
\begin{align*}
    \xymatrix{
    X\ar[r]^\xi\ar@{=}[d]& B\ar[r]^\beta\ar[d]^\varphi &C\ar[r]^\gamma\ar[d]^\phi &\Sigma X\ar@{=}[d]\\
    X\ar[r]_{\xi'} &E\ar[r]_{\epsilon}& D\ar[r]_{\delta}& \Sigma X.
    }
\end{align*}
Since $\End(X)$ is a division ring and $\xi$ is not a split monomorphism, we have that $B\in\widetilde{\mathcal{C}}\subset\mathcal{C}'$. Then, since $\xi':X\rightarrow E$ is left almost split in $\mathcal{C}'$, there is a morphism $\eta:E\rightarrow B$ such that $\eta\circ \xi'=\xi$. By the axioms of triangulated categories, we obtain a morphism of triangles of the form:
\begin{align*}
    \xymatrix{
    X\ar[r]^{\xi'}\ar@{=}[d]& E\ar[r]^{\epsilon}\ar[d]^\eta &D\ar[r]^{\delta}\ar[d]^\nu &\Sigma X\ar@{=}[d]\\
    X\ar[r]_\xi &B\ar[r]_\beta& C\ar[r]_\gamma& \Sigma X.
    }
\end{align*}
Consider the composition of these two triangle morphisms. As $\xi'$ is left minimal and $\xi'=\varphi\eta \xi'$, it follows that  $\varphi\eta: E\rightarrow E$ is an isomorphism. Then, by the $5$-Lemma \cite[exercise 10.2.2]{W}, we have that $\phi\nu:D\rightarrow D$ is an isomorphism. In particular, $\nu: D\rightarrow C$ is a split monomorphism and $D$ is a direct summand of $C$. As $C$ is indecomposable, this means that $D\cong C$, contradicting the fact that $D$ is in $\mathcal{C}'$ while $C$ is not. Hence $X$ is $\Ext$-injective.

(c) Suppose now that $\End(C)$ is a division ring, $\mathcal{C}'$ is closed under extensions and $X\in P(\mathcal{C}')$.
Since $\End(C)$ is a division ring, then every endomorphism of $C$ is either an isomorphism or it is the zero morphism. Moreover, as $\beta$ is not a split epimorphism, we have that $C$ is not a direct summand of $B$. Hence  $B\in\widetilde{\mathcal{C}}$ and it is easy to see that $\beta$ is a $\widetilde{C}$-cover of $C$.

Next, we prove that $\widetilde{\mathcal{C}}$ is closed under extensions. Consider a triangle with end terms $\widetilde{C''},\, \widetilde{C'}$ in $\widetilde{\mathcal{C}}$:
\begin{align*}
   \widetilde{C''}\rightarrow A \rightarrow \widetilde{C'}\rightarrow \Sigma \widetilde{C''}.
\end{align*}
Since $\mathcal{C}$ and $\mathcal{C}'$ are closed under extensions and $\widetilde{\mathcal{C}}$ is contained in both, then $A\in\mathcal{C}\cap \mathcal{C}'=\widetilde{\mathcal{C}}$ and $\widetilde{\mathcal{C}}$ is closed under extensions.
By Triangulated Wakamatsu's Lemma, see \cite[Lemma 2.1]{JP}, we then have that $\Hom (\widetilde{C}, \Sigma X)=0$ for all $\widetilde{C}\in\widetilde{\mathcal{C}}$.
Moreover, since $X\in P(\mathcal{C}')$ by assumption, we have that $\Hom (X, \Sigma X)=0$. Hence $\Hom (C', \Sigma X)=0$ for all $C'\in\mathcal{C}'$ and $X\in I(\mathcal{C}')$.
\end{proof}

\begin{remark}\label{remark_2CY_Ext}
If $\mathcal{T}$ is $2$-Calabi-Yau, letting $D(-)=\Hom_{k}(-,k)$, for any $X,\,Y\in\mathcal{T}$, we have that  $\Ext^1(X,Y)=\Hom(X,\Sigma Y)\cong D\Hom(Y,\Sigma X)=D\Ext^1(Y,X)$. Hence $\Ext$-projective and $\Ext$-injective objects coincide in additive subcategories of $\mathcal{T}$.
\end{remark}

\begin{corollary}\label{coro_mutation}
In the setup of Theorem \ref{thm_mutation}, suppose that $\mathcal{T}$ is $2$-Calabi-Yau and $\End(X)$ and $\End(C)$ are division rings. Then $\mathcal{C}'$ is closed under extensions if and only if $X\in I(\mathcal{C}')$.
\end{corollary}
\begin{proof}
Since $\mathcal{T}$ is $2$-Calabi-Yau, we have that $\Ext$-injective and $\Ext$-projective objects in $\mathcal{C}'$ coincide by Remark \ref{remark_2CY_Ext}. The result then follows directly from Theorem \ref{thm_mutation}.
\end{proof}

We state, without proof, the duals of Theorem \ref{thm_mutation} and Corollary \ref{coro_mutation}.

\begin{theorem}\label{thm_mutation_dual}
Assume $\mathcal{C}$ is functorially finite in $\mathcal{T}$ and $A\in I(\mathcal{C})$ is indecomposable. Then there is a right-weak Auslander-Reiten triangle in $\mathcal{C}$ of the form
\begin{align*}
A\xrightarrow{\alpha} B\xrightarrow{\beta} Z\xrightarrow{\zeta}\Sigma A.
\end{align*}
Let $\underline{\mathcal{C}}$ be the additive category with $\text{Ind }\underline{\mathcal{C}}=\text{Ind }(\mathcal{C})\setminus A$ and $\mathcal{C}'':=\text{add }(\underline{\mathcal{C}}\cup Z)$.
\begin{enumerate}[label=(\alph*)]
\item If $Z\in P(\mathcal{C}'')\cap I(\mathcal{C}'')$, then $\mathcal{C}''$ is closed under extensions.
\item If $\End(Z)$ and $\End(A)$ are division rings and $\mathcal{C}''$ is closed under extensions, then $Z\in P(\mathcal{C}'')$.
\item If $\End(A)$ is a division ring, $\mathcal{C}''$ is closed under extensions and $Z\in I(\mathcal{C}'')$, then $Z\in P(\mathcal{C}'')$.
\end{enumerate}
\end{theorem}

\begin{corollary}
In the setup of Theorem \ref{thm_mutation_dual}, suppose that $\mathcal{T}$ is $2$-Calabi-Yau and $\End(Z)$ and $\End(A)$ are division rings. Then $\mathcal{C}''$ is closed under extensions if and only if $Z\in P(\mathcal{C}'')$.
\end{corollary}

\section{Subcategories of the form $\mathcal{C}'$ and mutations of $\mathcal{C}$}\label{section_mutation_same}
As mentioned before, the idea of how to construct $\mathcal{C}'$ from $\mathcal{C}$, by removing the third term of a left-weak Auslander-Reiten triangle $\Delta$ in $\mathcal{C}$ and replacing it with the first term of $\Delta$, is similar to the classic mutation from \cite{IY}. In general, these two constructions are different. However, they have the same result under some extra assumptions, as we show in this section. 

\begin{defn}[{\cite[definition 3.1]{ZZ}}]
Let $\mathcal{D}\subseteq\mathcal{C}$ be an additive functorially finite rigid subcategory. For any object $C\in\mathcal{C}$, let $\delta:D\rightarrow C$ be a $\mathcal{D}$-cover and complete it to a triangle of the form $\mu_{\mathcal{D}}(C)\rightarrow D\xrightarrow{\delta}C\rightarrow \Sigma D$. Then $\mu_{\mathcal{D}}(C)$ is the \textit{backward $\mathcal{D}$-mutation of $C$} and the \textit{backward $\mathcal{D}$-mutation of $\mathcal{C}$} is
\begin{align*}
\mu(\mathcal{C};\mathcal{D}):=\text{add }(\{ \mu_{\mathcal{D}}(C)\mid C\in\mathcal{C} \}\cup\mathcal{D}).
\end{align*}
\end{defn}

\begin{lemma}\label{lemma_mutationZZ}
Assume that $\mathcal{T}$ is $2$-Calabi-Yau and let $\widetilde{\mathcal{C}}\subseteq \mathcal{T}$ be an additive subcategory closed under extensions which has finitely many indecomposable objects. Letting $\mathcal{D}=P(\widetilde{\mathcal{C}})$, we have that
\begin{align*}
    \mu(\widetilde{\mathcal{C}};\mathcal{D})=\widetilde{\mathcal{C}}.
\end{align*}
\end{lemma}
\begin{proof}
First note that since $\mathcal{C}$ has finitely many indecomposable objects, both $\mathcal{C}$ and $\widetilde{\mathcal{C}}$ are functorially finite in $\mathcal{T}$. Then, since $\widetilde{\mathcal{C}}$ is an extension closed subcategory functorially finite in $\mathcal{T}$, by the dual of \cite[proposition 2.3(1)]{IY}, there exists a cotorsion pair of the form $(\mathcal{X},\widetilde{\mathcal{C}})$.
Since $\mathcal{D}=P(\widetilde{\mathcal{C}})$, by the dual of \cite[proposition 3.7(3)]{ZZ}, we have that $\mu(\widetilde{\mathcal{C}};\mathcal{D})\subseteq\widetilde{\mathcal{C}}$. Moreover, by the dual of \cite[proposition 3.7(1)]{ZZ}, there is a bijection between $\text{Ind }\widetilde{\mathcal{C}}$ and $\text{Ind }\mu(\widetilde{\mathcal{C}};\mathcal{D})$ and, since these are finite sets, we conclude that $\mu(\widetilde{\mathcal{C}};\mathcal{D})=\widetilde{\mathcal{C}}$.
\end{proof}

\begin{lemma}\label{lemma_mutation_CisX}
In the setup of Theorem \ref{thm_mutation}, suppose that $\mathcal{T}$ is $2$-Calabi-Yau, $X\in P(\mathcal{C}')$ and $\End(C)$ is a division ring. Let $\mathcal{D}\subseteq \mathcal{T}$ be the additive subcategory generated by all the indecomposable $\Ext$-projectives in $\mathcal{C}$ apart from $C$ and note that $\mathcal{D}$ is rigid. Assume $\mathcal{D}$ is functorially finite in $\mathcal{C}$. Then $X\cong \mu_{\mathcal{D}}(C)$.
\end{lemma}
\begin{proof}
Consider the triangle $(\star)$ from Theorem \ref{thm_mutation}. Let $\delta: D\rightarrow C$ be a $\mathcal{D}$-cover and note that $\delta$ is not a split epimorphism since $C$ is not in $\mathcal{D}$. Then, since $\beta$ is right almost split in $\mathcal{C}$ and $D\in\mathcal{D}\subseteq\mathcal{C}$, it follows that $\delta$ factors through $\beta$ and we obtain a morphism of triangles of the form:
\begin{align*}
    \xymatrix{
    \mu_{\mathcal{D}}(C)\ar[r]\ar[d]^\eta& D\ar[r]^\delta\ar[d]^\varphi& C\ar[r]\ar@{=}[d]&\Sigma \mu_{\mathcal{D}}(C)\ar[d]\\
    X\ar[r]_\xi & B\ar[r]_{\beta}&C\ar[r]_\gamma &\Sigma X,
    }
\end{align*}
where $\eta$ exists by the axioms of triangulated categories. For $A\in\mathcal{C}$, consider the exact sequence:
\begin{align}
    \Hom(C,\Sigma A)\rightarrow \Hom(B,\Sigma A)\rightarrow \Hom(X,\Sigma A), \tag{\ddag}
\end{align}
and note that $\Hom(C,\Sigma A)=0$ since $C$ is $\Ext$-projective in $\mathcal{C}$. Without loss of generality, assume that $A$ is indecomposable. If $A\in\widetilde{\mathcal{C}}\subset\mathcal{C}'$, then $\Hom (X,\Sigma A)=0$ since $X$ is $\Ext$-projective in $\mathcal{C}'$. Then, exactness of $(\ddag)$ forces $\Hom(B,\Sigma A)=0$. If $A\not\in\widetilde{\mathcal{C}}$, then $A=C$ and since $C$ is $\Ext$-injective in $\mathcal{C}$, then $\Hom(B,\Sigma A)=0$. Hence $\Ext^1(B,A)=0$ for any $A$ in $\mathcal{C}$ and so $B$ is $\Ext$-projective in $\mathcal{C}$. Since $\End(C)$ is a division ring, we have that $C$ is not a direct summand of $B$ and so $B\in\mathcal{D}$. Note that as $\delta$ is a $\mathcal{D}$-cover, then $\beta$ factors through $\delta$ and, by the axioms of triangulated categories, we obtain a morphism of triangles of the form:
\begin{align*}
    \xymatrix{
    X\ar[r]^\xi\ar[d]^\nu& B\ar[r]^\beta\ar[d]^\phi& C\ar[r]\ar@{=}[d]&\Sigma X\ar[d]\\
    \mu_{\mathcal{D}}(C)\ar[r]&D\ar[r]_\delta&C\ar[r]&\Sigma \mu_{\mathcal{D}}(C).
    }
\end{align*}
Then $\delta=\delta\phi \varphi$ and, as $\delta$ is right minimal, then $\phi\varphi: D\rightarrow D$ is an isomorphism. By the $5$-Lemma, see \cite[exercise\ 10.2.2]{W}, it follows that $\nu\eta:\mu_{\mathcal{D}}(C)\rightarrow \mu_{\mathcal{D}}(C)$ is an isomorphism. Hence $\eta$ is a split epimorphism and $\mu_{\mathcal{D}}(C)$ is a direct summand of $X$. Since $X$ is indecomposable, it follows that $X\cong \mu_{\mathcal{D}}(C)$.
\end{proof}

\begin{theorem}\label{thm_mutZZ_same}
Assume $\mathcal{C}$ has finitely many indecomposables and $C\in P(\mathcal{C})$ is an indecomposable. Then there is a left-weak Auslander-Reiten triangle in $\mathcal{C}$ of the form
\begin{align*}
X\xrightarrow{\xi} B\xrightarrow{\beta} C\xrightarrow{\gamma}\Sigma X.
\end{align*}
Let $\widetilde{\mathcal{C}}\subseteq \mathcal{T}$ be the additive subcategory with $\text{Ind }\widetilde{\mathcal{C}}=\text{Ind }(\mathcal{C})\setminus C$ and define $\mathcal{C}':=\text{add }(\widetilde{\mathcal{C}}\cup X)$.
Suppose, moreover, that $\mathcal{T}$ is $2$-Calabi-Yau and $X\in P(\mathcal{C}')$. Then, letting $\mathcal{D}=\text{add }P(\widetilde{\mathcal{C}})$, we have
\begin{align*}
    \mu (\mathcal{C}; \mathcal{D})=\mathcal{C}',
\end{align*}
 and this is a subcategory of $\mathcal{T}$ closed under extensions.
\end{theorem}

\begin{proof}
First note that $X\in P(\mathcal{C'})$ implies $X\in P(\widetilde{\mathcal{C}})$, so $X\in I(\widetilde{\mathcal{C}})$ by Remark \ref{remark_2CY_Ext}. Hence $\widetilde{\mathcal{C}}$ is closed under extensions by Theorem \ref{thm_mutation}(a). By Lemma \ref{lemma_mutationZZ}, we have that $\mu(\widetilde{\mathcal{C}};\mathcal{D})=\widetilde{\mathcal{C}}$. Moreover, we have that $\text{Ind } P(\mathcal{C})=\text{Ind }P(\widetilde{\mathcal{C}})\cup \{ C \}$. In fact, if $D\in\text{Ind } P(\mathcal{C})$, then either $D=C$ or $D\in\widetilde{\mathcal{C}}$, in which case $\Ext^1(D,\widetilde{\mathcal{C}})=0$ since $\Ext^1(D,\mathcal{C})=0$ and $\widetilde{\mathcal{C}}\subset\mathcal{C}$. On the other hand, if $\widetilde{D}\in\text{Ind }P(\widetilde{\mathcal{C}})$, we have $\Ext^1(\widetilde{D}, \widetilde{\mathcal{C}})=0$ and, since $C$ is $\Ext$-injective in $\mathcal{C}$, also $\Ext^1(\widetilde{D},C)=0$ so that $\Ext^1(\widetilde{D},\mathcal{C})=0$.

So $\mathcal{D}=P(\widetilde{\mathcal{C}})$ is the additive category generated by all the indecomposable $\Ext$-projectives in $\mathcal{C}$ apart from $C$. Hence $\mathcal{D}$ is rigid in $\mathcal{C}$ and we can mutate $\mathcal{C}$ with respect to $\mathcal{D}$. Since $\mathcal{C}$ has finitely many indecomposables and $\mathcal{C}$ is Hom-finite, then $\mathcal{D}\subseteq \mathcal{C}$ is functorially finite. Then, by Lemma \ref{lemma_mutation_CisX}, we have that $\mu_{\mathcal{D}}(C)\cong X$, where $X$ is the first term of the triangle $(\star)$ from Theorem \ref{thm_mutation}. Hence, we conclude that 
\begin{align*}
    \mu(\mathcal{C};\mathcal{D})=\text{add }(\mu (\widetilde{\mathcal{C}};\mathcal{D})\cup \mu_{\mathcal{D}}(C))=\text{add }(\widetilde{\mathcal{C}}\cup X)=\mathcal{C}',
\end{align*}
and this subcategory of $\mathcal{T}$ is closed under extensions by Theorem \ref{thm_mutation}.
\end{proof}

We present the definition of forward $\mathcal{D}$-mutation and state, without proof, the dual of Theorem \ref{thm_mutZZ_same}.

\begin{defn}[{\cite[definition 3.1]{ZZ}}]
Let $\mathcal{D}\subseteq\mathcal{C}$ be an additive functorially finite rigid subcategory. For any object $A\in\mathcal{C}$, let $\alpha:A\rightarrow D$ be a $\mathcal{D}$-envelope and complete it to a triangle of the form $A\xrightarrow{\alpha} D\rightarrow \mu^{-1}_{\mathcal{D}}(A)\rightarrow\Sigma A$. Then $\mu^{-1}_{\mathcal{D}}(A)$ is the \textit{forward $\mathcal{D}$-mutation of $A$} and we define the \textit{forward $\mathcal{D}$-mutation of $\mathcal{C}$} to be
\begin{align*}
\mu^{-1}(\mathcal{C};\mathcal{D}):=\text{add }(\{ \mu^{-1}_{\mathcal{D}}(A)\mid A\in\mathcal{C} \}\cup\mathcal{D}).
\end{align*}
\end{defn}

\begin{theorem}
In the setup of Theorem \ref{thm_mutation_dual}, suppose that $\mathcal{T}$ is $2$-Calabi-Yau, $\mathcal{C}$ has finitely many indecomposables and $Z\in I(\mathcal{C}'')$. Then, letting $\mathcal{D}=\text{add }I(\underline{\mathcal{C}})$, we have
\begin{align*}
    \mu^{-1}(\mathcal{C};\mathcal{D})=\mathcal{C}'',
\end{align*}
and this is a subcategory of $\mathcal{T}$ closed under extensions.
\end{theorem}

\section{Example: cluster category of Dynkin type $A_n$}\label{section6}
Given a finite quiver with no oriented cycles $Q$, and letting $k=\mathbb{C}$, one can define the cluster category of $Q$, usually denoted $\mathcal{C}_Q$, see \cite{BMRRT}. Then $\mathcal{T}=\mathcal{C}_Q$ satisfies Setup \ref{setup} and it is $2$-Calabi-Yau, see \cite{K} and \cite{K2} for more details. Here we focus on the case when $\mathcal{T}$ is the cluster category of Dynkin type $A_n$, \textbf{i.e.} when $\mathcal{T}=\mathcal{C}_{A_n}$ for
\begin{align*}
Q=A_n: \,\, \stackrel{1}{\bullet}\longleftarrow \stackrel{2}{\bullet}\longleftarrow \cdots \longleftarrow \stackrel{n-1}{\bullet}\longleftarrow \stackrel{n}{\bullet}.
\end{align*}
We study $\mathcal{T}$ through a geometric realisation of it. Let $P$ be the regular polygon with $n+3$ vertices. The following can be proved using the results in \cite{CCS}.

\begin{enumerate}[itemsep=2ex, label=(\Roman*)]
\item\label{cluster1} There is a bijection
\begin{align*}
\{\text{diagonals in $P$ between non-neighbouring vertices} \}\leftrightarrow \text{Ind }\mathcal{T}.
\end{align*}
We identify $\text{Ind }\mathcal{T}$ and the diagonals of $P$, so given an indecomposable $x\in\mathcal{T}$ it makes sense to write $x=\{x_{0},x_{1}\}$, for $x_{0},x_{1}$ its endpoints as a diagonal in $P$. 
\item \label{cluster2} Let the diagonals $\mathfrak{a}, \mathfrak{c}$ correspond respectively to the indecomposables $a,c$ under the bijection from \ref{cluster1}. Then 
\begin{align*}
\dim_{\C} (\Ext^1(a,c))= \dim_{\C}  (\Hom (a,\Sigma c))=
\begin{cases}
1 & \text{if } \mathfrak{a},\mathfrak{c} \text{ cross}, \\ 0 & \text{otherwise,} 
\end{cases}
\end{align*}
where we say that two diagonals \textit{cross} if they intersect in the interior of $P$ (so excluding the endpoints).
\item If $a\in \text{Ind }\mathcal{T}$ corresponds to the diagonal $\mathfrak{a}=\{a_{0}, a_{1}\}$, then $\Sigma a$ corresponds to the diagonal $\{a_{0}^-,a_{1}^- \}$ obtained by moving the endpoints of $\mathfrak{a}$ by one clockwise step, see Figure \ref{fig:sigma_a}.
     \begin{figure}
  \centering
    \begin{tikzpicture}[scale=2]
      \draw (0,0) circle (1cm); 
      
     \draw (110:0.97cm) -- (110:1.03cm);
     \draw (90:0.97cm) -- (90:1.03cm);
     \draw (70:0.97cm) -- (70:1.03cm);
     \draw (50:0.97cm) -- (50:1.03cm);
     \draw (30:0.97cm) -- (30:1.03cm);
     \draw (10:0.97cm) -- (10:1.03cm);
     \draw (-49:0.97cm) -- (-49:1.03cm);
     \draw (-69:0.97cm) -- (-69:1.03cm);
     \draw (-89:0.97cm) -- (-89:1.03cm);
     \draw (-109:0.97cm) -- (-109:1.03cm);
     \draw (-129:0.97cm) -- (-129:1.03cm);
     \draw (-149:0.97cm) -- (-149:1.03cm);
     \draw (-169:0.97cm) -- (-169:1.03cm);
     \draw (-189:0.97cm) -- (-189:1.03cm);
      \draw (-9:0.97cm) -- (-9:1.03cm);
      \draw (-9:1.13cm) node{$\scriptstyle a_{1}$};
      
      \draw (150:0.97cm) -- (150:1.03cm);
      \draw (150:1.13cm) node{$\scriptstyle a_{0}$};
      
      \draw (-29:0.97cm) -- (-29:1.03);
      \draw (-29:1.18cm) node{$\scriptstyle a_{1}^{-}$};

      \draw (130:0.97cm) -- (130:1.03cm); 
      \draw (130:1.18cm) node{$\scriptstyle a_{0}^{-}$};
      
      \draw (150:1cm) -- (-9:1cm);
      \draw (-29:1cm) -- (130:1cm);
      \draw[very thick, orange, ->] ([shift=(150:1.3cm)]0,0) arc (150:135:1.4cm);
      \draw[very thick, orange, ->] ([shift=(-11:1.3cm)]0,0) arc (-11:-26:1.4cm);
     
    \end{tikzpicture}
    \caption{$\mathfrak{a}=\{a_{0}, a_{1}\}$ and $\Sigma \mathfrak{a}=\{a_{0}^-,a_{1}^- \}$ .}
    \label{fig:sigma_a}
\end{figure}
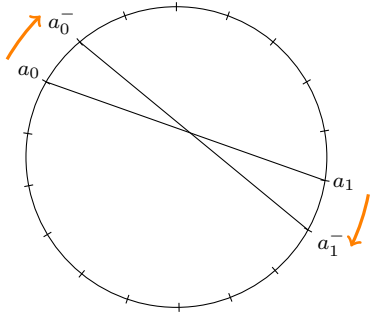

\item\label{cluster4} In \ref{cluster2}, suppose $\mathfrak{a},\mathfrak{c}$ cross, so $\dim_{\C}  (\Hom (a,\Sigma c))=\dim_{\C}  (\Hom (c,\Sigma a))=1$.
Then we can complete the non-zero morphisms $c\rightarrow \Sigma a$ and $a\rightarrow\Sigma c$ to obtain the two triangles
\begin{align*}
a\rightarrow b_1 \oplus b_2\rightarrow c\rightarrow \Sigma a\\
c\rightarrow s_1 \oplus s_2\rightarrow a\rightarrow \Sigma c,
\end{align*}
where $b_1,b_2,s_{1},s_{2}$ are the indecomposables corresponding to the diagonals $\mathfrak{b}_{1},\mathfrak{b}_{2},\mathfrak{s}_{1},\mathfrak{s}_{2}$ respectively in Figure \ref{fig:figtriangle}.
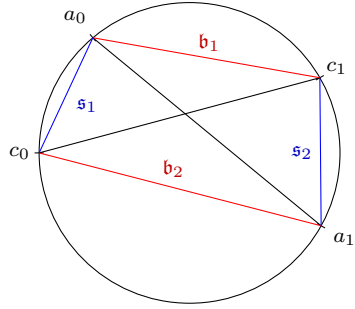
\begin{figure}
  \centering
    \begin{tikzpicture}[scale=2]
      \draw (0,0) circle (1cm);

      \draw (30:0.97cm) -- (30:1.03cm);
      \draw (30:1.13cm) node{$\scriptstyle c_{1}$};
      
      \draw (180:0.97cm) -- (180:1.03cm);
      \draw (180:1.13cm) node{$\scriptstyle c_{0}$};
      
      \draw (-29:0.97cm) -- (-29:1.03);
      \draw (-29:1.18cm) node{$\scriptstyle a_{1}$};

      \draw (130:0.97cm) -- (130:1.03cm); 
      \draw (130:1.18cm) node{$\scriptstyle a_{0}$};
      
      \draw (180:1cm) -- (30:1cm);
      \draw (-29:1cm) -- (130:1cm);
     
     \draw [red] (-29:1cm) -- (180:1cm);
     \draw [red] (30:1cm) -- (130:1cm);
     \draw [blue] (130:1cm) -- (180:1cm);
     \draw [blue] (30:1cm) -- (-29:1cm);
     
     \draw [color=red!75!black] (80:0.75cm) node{$\scriptstyle \mathfrak{b}_{1}$};
     \draw [color=red!75!black] (-136:0.15cm) node{$\scriptstyle \mathfrak{b}_{2}$};
     \draw [color=blue!75!black] (155:0.75cm) node{$\scriptstyle \mathfrak{s}_{1}$};
      \draw [color=blue!75!black] (0:0.75cm) node{$\scriptstyle \mathfrak{s}_{2}$};
    \end{tikzpicture}
    \caption{There are triangles $a\rightarrow b_{1}\oplus b_{2}\rightarrow c\rightarrow$ and $c\rightarrow s_1 \oplus s_2\rightarrow a\rightarrow \Sigma c$ in $\mathcal{C}_{A_n}$.}
    \label{fig:figtriangle}
\end{figure}
If $s_i$ or $b_i$ corresponds to an edge of $P$, then it is zero in $\mathcal{T}$.
\item\label{cluster5}
Of the triangles from \ref{cluster4}, the Auslander-Reiten triangles are exactly those in which either $a=\Sigma c$ or $c=\Sigma a$. For example, if $a=\Sigma c$, then we have an Auslander-Reiten triangle and a trivial triangle, respectively
\begin{align*}
\Sigma c\rightarrow b_1 \oplus b_2\rightarrow c\rightarrow \Sigma^{2} c \text{ and }
c\rightarrow 0 \rightarrow \Sigma c\xrightarrow{\cong} \Sigma c.
\end{align*}
Note that in this case $\mathfrak{s}_{1},\mathfrak{s}_{2}$ are edges of $P$ and hence zero in $\mathcal{T}$.
\item\label{cluster6}
Labelling the vertices of $P$ from $0$ to $n+2$ anticlockwise and using \ref{cluster1}-\ref{cluster5}, the Auslander-Reiten quiver of $\mathcal{T}$ is as shown in Figure \ref{fig:diagonals}. Note that this can be drawn on a M\"{o}bius strip.
For a diagonal $\mathfrak{a}$, we have that $\Sigma \mathfrak{a}$ is placed in the same row and to the left of $\mathfrak{a}$ and the Auslander-Reiten triangle $\Sigma a\rightarrow s_1\oplus s_2\rightarrow a\rightarrow \Sigma^2 a$ corresponds to the mesh:
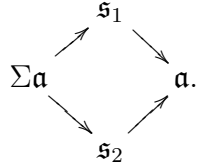
\begin{align*}
    \xymatrix@R=1em @C=1em {
    &\mathfrak{s}_1\ar[rd]&\\
    \Sigma \mathfrak{a}\ar[ru]\ar[rd]& & \mathfrak{a}.\\
    &\mathfrak{s}_2 \ar[ru]&
    }
\end{align*}
\begin{figure}
  \centering
    \begin{tikzpicture}[scale=1]
\draw (0,0)  node{$\scriptstyle \{0,2\}$};
\draw (1,1)  node{$\scriptstyle \{0,3\}$};
\draw (2,2)  node{$\scriptstyle \bullet$};
\draw (3,3)  node{$\scriptstyle \{0,n\}$};
\draw (4,4)  node{$\scriptstyle \{0,n+1\}$};

\draw (2,0)  node{$\scriptstyle \{1,3\}$};
\draw (3,1)  node{$\scriptstyle \{1,4\}$};
\draw (5,1)  node{$\scriptstyle \bullet$};
\draw (6,0)  node{$\scriptstyle \bullet$};

\draw (5,3)  node{$\scriptstyle \{1,n+1\}$};
\draw (6,2)  node{$\scriptstyle \bullet$};
\draw (7,1)  node{$\scriptstyle \{n-2,n+1\}$};
\draw (8,0)  node{$\scriptstyle \{n-1,n+1\}$};

\draw [->] (0.2,0.2)--(0.8,0.8);
\draw [->] (1.2,1.2)--(1.8,1.8);
\draw [dashed] (2.2,2.2)--(2.8,2.8);
\draw [->] (3.2,3.2)--(3.8,3.8);
\draw [->](4.2,3.8)--(4.8,3.2);
\draw[dashed](5.2,2.8)--(5.8,2.2);
\draw [->](6.2,1.8)--(6.8,1.2);
\draw [->](7.2,0.8)--(7.8,0.2);

\draw[very thick, dotted](3.7,1)--(4.3,1);
\draw[very thick, dotted](3.7,0)--(4.3,0);

\draw [->] (1.2,0.8)--(1.8,0.2);
\draw [->] (2.2,0.2)--(2.8,0.8);
\draw [->] (2.2,1.8)--(2.8,1.2);
\draw [->] (5.2,1.2)--(5.8,1.8);
\draw [->] (5.2,0.8)--(5.8,0.2);
\draw [->] (6.2,0.2)--(6.8,0.8);

\draw (6,4)  node{$\scriptstyle \{1,n+2\}$};
\draw (7,3)  node{$\scriptstyle \{2,n+2\}$};
\draw (8,4)  node{$\scriptstyle \{0,2\}$};
\draw (9,3)  node{$\scriptstyle \{0,3\}$};
\draw (10,2)  node{$\scriptstyle \bullet$};
\draw (8,2)  node{$\scriptstyle \bullet$};
\draw (9,1)  node{$\scriptstyle \{n-1,n+2\}$};
\draw (10,0)  node{$\scriptstyle \{n,n+2\}$};
\draw (11,1)  node{$\scriptstyle \{0,n\}$};
\draw (12,0)  node{$\scriptstyle \{0,n+1\}$};

\draw [->] (5.2,3.2)--(5.8,3.8);
\draw [->] (6.2,3.8)--(6.8,3.2);
\draw [->] (7.2,3.2)--(7.8,3.8);
\draw [dashed] (7.2,2.8)--(7.8,2.2);
\draw[->] (8.2,1.8)--(8.8,1.2);
\draw[->] (9.2,0.8)--(9.8,0.2);
\draw[->] (7.2,1.2)--(7.8,1.8);
\draw[->] (8.2,0.2)--(8.8,0.8);

\draw[->] (10.2,0.2)--(10.8,0.8);

\draw[->](8.2,3.8)--(8.8,3.2);
\draw[dashed](9.2,2.8)--(9.8,2.2);
\draw[->](10.2,1.8)--(10.8,1.2);
\draw[->](9.2,1.2)--(9.8,1.8);
\draw[->] (11.2,0.8)--(11.8,0.2);

  \end{tikzpicture}
 \caption{Auslander-Reiten quiver of $\mathcal{C}_{A_n}$.}
    \label{fig:diagonals}
\end{figure}
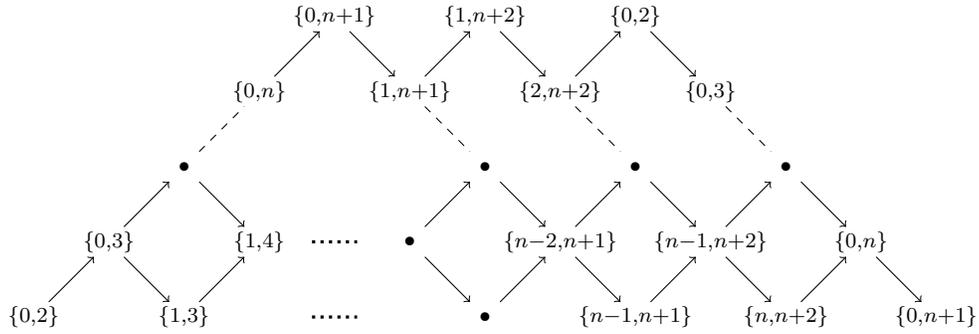

\item\label{cluster7}
We can define a cyclic order on the vertices of $P$ as follows. Given three vertices $u,v,w$ of $P$, we write $u<v<w$ if they appear in the order $u,v,w$ when going through the vertices of $P$ in the positive direction. Moreover, if we choose two distinct vertices $u$ and $v$, we can consider the interval of vertices $[u,v]$ and in this ``$<$'' is a total order.
\item\label{cluster8}
Let $x=\{x_{0},x_{1}\}\in \text{Ind }\mathcal{T}$. Then, by \cite[lemma\ 2.4.2]{IT}, we have that $y=\{y_{0},y_{1}\}\in \text{Ind }\mathcal{T}$ is such that $\Hom(x,y)\neq 0$ if and only if $y$ has one endpoint in each of the intervals $[x_0,x_1^{--}]$ and $[x_1,x_0^{--}]$, \textbf{i.e.} the blue arcs in Figure \ref{fig:fig_morphisms_from_x}.

\noindent Moreover, for such a $y$, the indecomposables $s=\{s_{0},s_{1}\}$ such that the morphism $x\rightarrow y$ factors through $s$ are exactly those having one endpoint in each of the intervals $[x_0,y_0]$ and $[x_1,y_1]$, \textbf{i.e.} the two red arcs in Figure \ref{fig:fig_morphisms_from_x}.
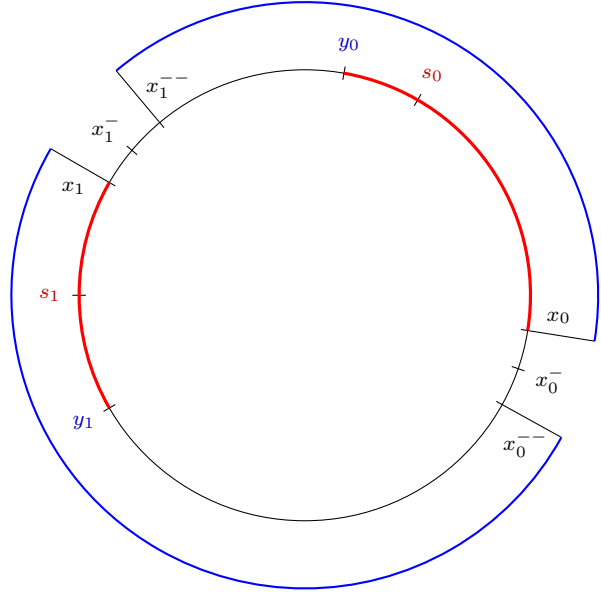
\begin{figure}
\centering
\begin{tikzpicture}[scale=3]
   \draw (0,0) circle (1cm);
   
   \draw (-9:0.97cm) -- (-9:1.3cm);
   \draw (-5:1.13cm) node{$\scriptstyle x_{0}$};
   \draw (150:0.97cm) -- (150:1.3cm);
   \draw (155:1.13cm) node{$\scriptstyle x_{1}$};
   
   \draw (-19:0.97cm) -- (-19:1.03);
   \draw (-19:1.15cm) node{$\scriptstyle x_{0}^{-}$}; 
   \draw (140:0.97cm) -- (140:1.03cm); 
   \draw (140:1.15cm) node{$\scriptstyle x_{1}^{-}$};
   \draw (-29:0.97cm) -- (-29:1.3);
   \draw (-34:1.18cm) node{$\scriptstyle x_{0}^{--}$};
   \draw (130:0.97cm) -- (130:1.3cm); 
   \draw (123:1.11cm) node{$\scriptstyle x_{1}^{--}$};
   
   \draw (80:0.97cm) -- (80:1.03cm);
   \draw [color=blue!75!black] (80:1.13cm) node{$\scriptstyle y_{0}$};
   \draw (-150:0.97cm) -- (-150:1.03cm);
   \draw [color=blue!75!black] (-150:1.13cm) node{$\scriptstyle y_{1}$};
   
   \draw (60:0.97cm) -- (60:1.03);
   \draw [color=red!75!black] (60:1.13cm) node{$\scriptstyle s_{0}$};
   \draw (180:0.97cm) -- (180:1.03);
   \draw [color=red!75!black] (180:1.13cm) node{$\scriptstyle s_{1}$}; 
   
   \draw[very thick, red] ([shift=(80:1cm)]0,0) arc (80:-9:1cm);
   \draw[very thick, red] ([shift=(-210:1cm)]0,0) arc (-210:-150:1cm);   
   
  
   \draw[thick, blue] ([shift=(130:1.3cm)]0,0) arc (130:-9:1.3cm);
   \draw[thick, blue] ([shift=(-210:1.3cm)]0,0) arc (-210:-29:1.3cm);
\end{tikzpicture} 
\caption{There is a non-zero morphism $x=\{ x_0,x_1 \}\rightarrow \{ y_0,y_1 \}=y$ if and only if $y$ has one endpoint below each blue arc. Moreover, $x\rightarrow y$ factors through $s=\{ s_0,s_1 \}$ if and only if $s$ has an endpoint below each red arc.}
\label{fig:fig_morphisms_from_x}
\end{figure}

\item
\label{cluster9}
Let $x=\{ x_{0},x_{1} \}\in\text{Ind }\mathcal{T}$. Then, by \cite[lemma\ 2.4.2]{IT}, we have that $z=\{ z_{0},z_{1} \}\in\text{Ind }\mathcal{T}$ is such that $\Hom(z,x)\neq 0$ if and only if $z$ has one endpoint in each of the intervals $[x_0^{++},x_1]$ and $[x_1^{++}, x_0]$, \textbf{i.e.} the two green arcs in Figure \ref{fig:fig_morphisms_to_x}.

\noindent Moreover, for such a $z$, the indecomposables $s=\{ s_0,s_1 \}$ such that the morphism $z\rightarrow x$ factors through $s$ are exactly those having one endpoint in each of the intervals $[z_0,x_1]$ and $[z_1,x_0]$, \textbf{i.e.} the two red arcs in Figure \ref{fig:fig_morphisms_to_x}.
\end{enumerate}

\begin{figure}
\centering
\begin{tikzpicture}[scale=3]
   \draw (0,0) circle (1cm);
   
   \draw (-9:0.97cm) -- (-9:1.3cm);
   \draw (-13:1.13cm) node{$\scriptstyle x_{0}$};
   \draw (150:0.97cm) -- (150:1.3cm);
   \draw (146:1.13cm) node{$\scriptstyle x_{1}$};
   
   \draw (1:0.97cm) -- (1:1.03);
   \draw (1:1.15cm) node{$\scriptstyle x_{0}^{+}$}; 
   \draw (160:0.97cm) -- (160:1.03cm); 
   \draw (160:1.15cm) node{$\scriptstyle x_{1}^{+}$};
   \draw (11:0.97cm) -- (11:1.3);
   \draw (16:1.15cm) node{$\scriptstyle x_{0}^{++}$};
   \draw (170:0.97cm) -- (170:1.3cm); 
   \draw (176:1.11cm) node{$\scriptstyle x_{1}^{++}$};
   
   \draw (80:0.97cm) -- (80:1.03cm);
   \draw [color=Green!75!black] (80:1.13cm) node{$\scriptstyle z_{0}$};
   \draw (-150:0.97cm) -- (-150:1.03cm);
   \draw [color=Green!75!black] (-150:1.13cm) node{$\scriptstyle z_{1}$};

  
   \draw[thick, Green] ([shift=(150:1.3cm)]0,0) arc (150:11:1.3cm);
   \draw[thick, Green] ([shift=(170:1.3cm)]0,0) arc (170:351:1.3cm);
   
   \draw (-60:0.97cm) -- (-60:1.03);
   \draw [color=red!75!black] (-60:1.13cm) node{$\scriptstyle s_{1}$};
   \draw (120:0.97cm) -- (120:1.03);
   \draw [color=red!75!black] (120:1.13cm) node{$\scriptstyle s_{0}$}; 
   
   \draw[very thick, red] ([shift=(-9:1cm)]0,0) arc (-9:-150:1cm);
   \draw[very thick, red] ([shift=(80:1cm)]0,0) arc (80:150:1cm); 
\end{tikzpicture} 
\caption{There is a non-zero morphism $z=\{ z_0,z_1 \}\rightarrow\{ x_0,x_1 \}=x$ if and only if $z$ has one endpoint below each green arc. Moreover $z\rightarrow x$ factors through $s=\{ s_0,s_1 \}$ if and only if $s$ has an endpoint below each red arc.}
\label{fig:fig_morphisms_to_x}
\end{figure}
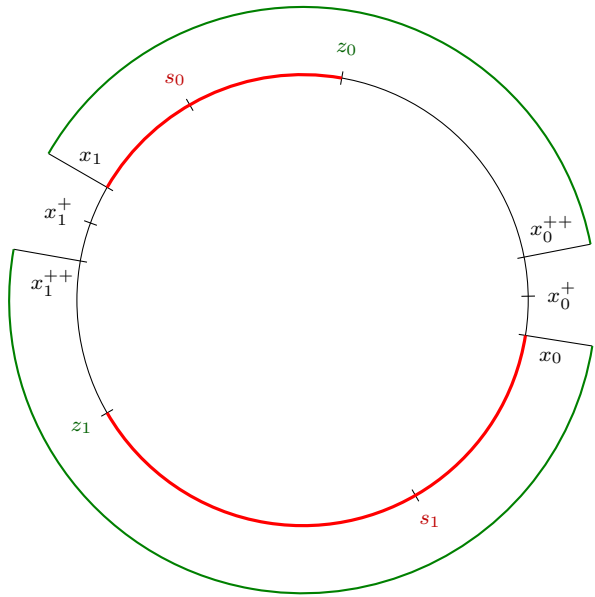

\subsection{Subcategories arising from Ptolemy diagrams}
The additive subcategories of $\mathcal{T}$ closed under extensions are precisely those arising from \textit{Ptolemy diagrams} in our regular $(n+3)$-gon $P$. The Ptolemy condition was first introduced in \cite{Ng} for the geometric representation of the cluster category of Dynkin type $A_\infty$ and later studied for the finite case in \cite{HJR}.
\begin{defn}
[{\cite[definition\ 2.1]{HJR}}]
A \textit{Ptolemy diagram} is a set $S$ of diagonals of a finite polygon such that if the set contains crossing diagonals $a$ and $b$, then it also contains all the diagonals connecting the endpoints of $a$ and $b$.
\end{defn}

Note that if we take $S$ to be the empty set, then this is a Ptolemy diagram, called an \textit{empty cell}. The set of all diagonals in a given polygon is also a Ptolemy diagram, called a \textit{clique}.

\begin{remark}[{\cite[theorem\ A (ii)]{HJR}}]
\label{thm_ptolemy}
Each Ptolemy diagram can be obtained by gluing empty cells and cliques.
\end{remark}

So, for our polygon $P$, a Ptolemy diagram is constructed by first choosing a set of pairwise non-crossing diagonals, called \textit{dissecting diagonals}, that divide $P$ in \textit{cells}, and then deciding whether each cell is empty or a clique.

\begin{exmp}
For example, in Figure \ref{fig:fig_ptolemy_eg} we have chosen the three green diagonals to be the dissecting diagonals and, going left to right, the first and third cells are empty and the second and fourth are cliques.
\end{exmp}
\begin{figure}
  \centering
    \begin{tikzpicture}[scale=3]
      \draw (0,0) circle (1cm); 
     \draw (110:0.97cm) -- (110:1.03cm);
     \draw (90:0.97cm) -- (90:1.03cm);
     \draw (70:0.97cm) -- (70:1.03cm);
     \draw (50:0.97cm) -- (50:1.03cm);
     \draw (30:0.97cm) -- (30:1.03cm);
     \draw (10:0.97cm) -- (10:1.03cm);
     \draw (-50:0.97cm) -- (-50:1.03cm);
     \draw (-70:0.97cm) -- (-70:1.03cm);
     \draw (-90:0.97cm) -- (-90:1.03cm);
     \draw (-110:0.97cm) -- (-110:1.03cm);
     \draw (-130:0.97cm) -- (-130:1.03cm);
     \draw (-150:0.97cm) -- (-150:1.03cm);
     \draw (-170:0.97cm) -- (-170:1.03cm);
     \draw (-190:0.97cm) -- (-190:1.03cm);
     \draw (-10:0.97cm) -- (-10:1.03cm);
     \draw (-30:0.97cm) -- (-30:1.03cm);
     \draw (150:0.97cm) -- (150:1.03cm);
     \draw (130:0.97cm) -- (130:1.03cm);
     
      \draw[thick, Green] (110:1cm) -- (-70:1cm);
      \draw[thick, Green] (-70:1cm) -- (50:1cm);
      \draw[thick, Green] (130:1cm) -- (-110:1cm);
      
     \draw[blue] (130:1cm) -- (-70:1cm);
     \draw[blue] (130:1cm) -- (-90:1cm);
     \draw[blue] (-90:1cm) -- (110:1cm);
     \draw[blue] (-110:1cm) -- (110:1cm);
     \draw[blue] (-110:1cm) -- (-70:1cm);

     \draw[blue] (-50:1cm) -- (30:1cm);
     \draw[blue] (-30:1cm) -- (30:1cm);
     \draw[blue] (-10:1cm) -- (30:1cm);
     \draw[blue] (-70:1cm) -- (30:1cm);
     
     \draw[blue] (50:1cm) -- (10:1cm);
     \draw[blue] (50:1cm) -- (-10:1cm);
     \draw[blue] (50:1cm) -- (-30:1cm);
     \draw[blue] (50:1cm) -- (-50:1cm);
     
     \draw[blue] (10:1cm) -- (-30:1cm);
     \draw[blue] (10:1cm) -- (-50:1cm);
     \draw[blue] (10:1cm) -- (-70:1cm);
     
     \draw[blue] (-10:1cm) -- (-50:1cm);
     \draw[blue] (-10:1cm) -- (-70:1cm);
     
     \draw[blue] (-30:1cm) -- (-70:1cm);
\end{tikzpicture}
    \caption{Example of a Ptolemy diagram.}
  \label{fig:fig_ptolemy_eg}
\end{figure}
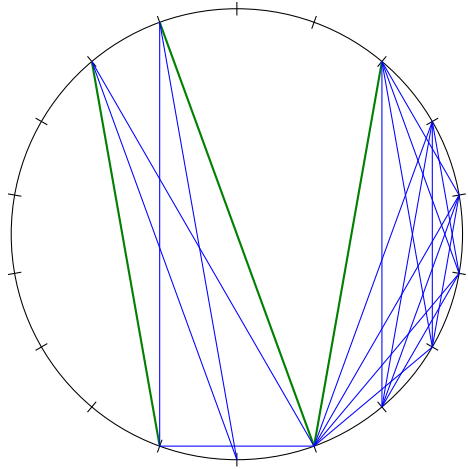

From now on let $\mathcal{C}$ be a subcategory of $\mathcal{T}$ corresponding to a Ptolemy diagram of $P$.
The following result is the reason why this choice satisfies Setup \ref{setup}.

\begin{proposition}\label{thm_ptolemy_ext}
Ptolemy diagrams of $P$ correspond to the additive subcategories of $\mathcal{T}$ closed under extensions.
\end{proposition}
\begin{proof}
This follows from \cite[theorem\ A(i) and proposition\ 2.3]{HJR}.
\end{proof}

\subsection{Minimal right almost split morphisms in $\mathcal{C}$ ending at Ext-projectives}\label{section35}
Note that since $\mathcal{C}\subseteq \mathcal{T}$ has finitely many indecomposables, it is functorially finite and we can apply Theorem \ref{prop_mras}. We first describe the indecomposable $\Ext$-projectives in $\mathcal{C}$. Then, given any indecomposable $\Ext$-projective in $\mathcal{C}$, we give a way to find a minimal right almost split morphism in $\mathcal{C}$ ending at it and the left-weak Auslander-Reiten triangle in $\mathcal{C}$ completing it.

\begin{proposition}\label{thm-extproj}
An indecomposable $c$ in $\mathcal{C}$ is $\Ext$-projective if and only if it is a dissecting diagonal.
\end{proposition}
\begin{proof}
Suppose $c$ is a dissecting diagonal, so there is no diagonal in $\mathcal{C}$ crossing it. Hence for every $a\in \mathcal{C}$, we have $\dim_{\C}(\Ext^1(a,c))=0$ and $c$ is $\Ext$-projective.

Suppose now that $c$ is not a dissecting diagonal and let $P'$ be the clique it belongs to. Then there are vertices of $P'$ lying on both sides of $c$. Joining any two vertices of $P'$ lying on different sides of $c$, we obtain a diagonal $a'$ in $\mathcal{C}$ crossing $c$  and hence $\Ext^1(c, a')\neq 0$. So $c$ is not $\Ext$-projective.
\end{proof}

\begin{setup} \label{setup_thm1}
Let $\mathcal{C}$ correspond to a Ptolemy diagram and $c\in \mathcal{C}$ be $\Ext$-projective and indecomposable. Then $c$ is a dissecting diagonal by Proposition \ref{thm-extproj}. Let the vertices of the two cells bordered by $c$ be $v_1<v_2<\cdots<v_m$ and $c=\{v_i,v_j\}$. Set $v_0:=v_m$.

Choose $v_p$ maximal in $[v_i^+,v_{j}^-]$ such that $b_0:=\{v_i, v_p\} \in\mathcal{C}$
and $v_q$ maximal in $[v_j^+,v_i^-]$ such that $b_1:= \{ v_j, v_q \}\in \mathcal{C}$. An example is shown in Figure \ref{fig:ext_proj}.


\end{setup}

\begin{remark}
Note that the choice of $v_p$ depends on whether $c$ borders a clique or an empty cell in $[v_i, v_j]$. In the first case we have $v_p=v_{j-1}$ while in the second $v_p=v_{i+1}$. Note that in the case when $c$ borders an empty cell with $v_{i+1}=v_i^+$, then $b_0=\{v_i,v_i^+\}=0$. The vertex $v_q$ is determined in a similar way, looking at the interval $[v_j,v_i]$.
\end{remark}

\begin{proposition}\label{thm_mras}
In the situation of Setup \ref{setup_thm1}, let
\begin{align*}
\beta=(\beta_0,\beta_1): b_0\oplus b_1\rightarrow c,
\end{align*}
where $\beta_0,\, \beta_1$ are non-zero morphisms (unless $b_0$ or $b_1$ are zero). Then $\beta$ is a minimal right almost split morphism in $\mathcal{C}$ and
\begin{align}
x\xrightarrow{\xi=\begin{psmallmatrix}\xi_0\\\xi_1\end{psmallmatrix}} b_0\oplus b_1\xrightarrow{\beta=(\beta_0,\beta_1)} c\rightarrow \Sigma x \tag{$\dag$}
\end{align}
is a triangle in $\mathcal{T}$ with $x=\{ v_p,v_q \}$ indecomposable not in $\mathcal{C}$ and $\xi$ a $\mathcal{C}$-envelope of $x$. In other words, $(\dag)$ is a left-weak Auslander-Reiten triangle in $\mathcal{C}$.
\end{proposition}

\begin{remark} In the example illustrated in Figure \ref{fig:ext_proj}, we have $x=\{v_4,v_2\}$. Note that $x$ crosses the dissecting diagonal $c$ and so it is not in $\mathcal{C}$.
\end{remark}
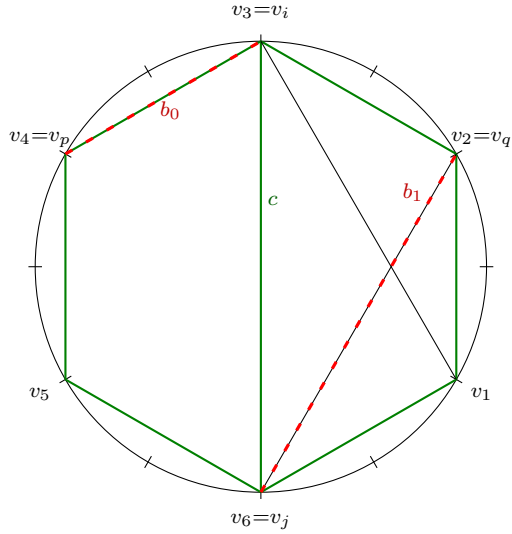
\begin{figure}
  \centering
    \begin{tikzpicture}[scale=3]
     \draw (0,0) circle (1cm); 
     \draw (0:0.97cm) -- (0:1.03cm);
     \draw (30:0.97cm) -- (30:1.03cm);
     \draw (60:0.97cm) -- (60:1.03cm);
     \draw (90:0.97cm) -- (90:1.03cm);
     \draw (120:0.97cm) -- (120:1.03cm);
     \draw (150:0.97cm) -- (150:1.03cm);
     \draw (180:0.97cm) -- (180:1.03cm);
     \draw (-30:0.97cm) -- (-30:1.03cm);
     \draw (-60:0.97cm) -- (-60:1.03cm);
     \draw (-90:0.97cm) -- (-90:1.03cm);
     \draw (-120:0.97cm) -- (-120:1.03cm);
     \draw (-150:0.97cm) -- (-150:1.03cm);
     
     \draw [Green, thick] (90:1cm) -- (-90:1cm);
     \draw  [Green, thick] (90:1cm) -- (30:1cm);
     \draw  [Green, thick] (30:1cm) -- (-30:1cm);
     \draw  [Green, thick] (-90:1cm) -- (-150:1cm);
     \draw  [Green, thick] (-150:1cm) -- (150:1cm);
     \draw [Green, thick] (150:1cm) -- (90:1cm);
     \draw [Green, thick] (-30:1cm) -- (-90:1cm);
     \draw [very thick, loosely dashed, red] (150:1cm) -- (90:1cm);
     \draw (90:1cm) -- (-30:1cm);
     \draw (-90:1cm) -- (30:1cm);
     \draw [very thick, loosely dashed, red] (-90:1cm) -- (30:1cm);
     
     \draw [color=Green!75!black](80:0.3cm) node{$\scriptstyle c$};
     \draw (90:1.13cm) node{$\scriptstyle v_{3}=v_i$};
     \draw (30:1.13cm) node{$\scriptstyle v_{2}=v_q$};
     \draw (-30:1.13cm) node{$\scriptstyle v_{1}$};
     \draw (150:1.13cm) node{$\scriptstyle v_{4}=v_p$};
     \draw (-150:1.13cm) node{$\scriptstyle v_{5}$};
     \draw (-90:1.13cm) node{$\scriptstyle v_6=v_j$};
     \draw [color=red!75!black](120:0.8cm) node{$\scriptstyle b_0$};
     \draw [color=red!75!black](26:0.75cm) node{$\scriptstyle b_1$};
     
\end{tikzpicture}
    \caption{Example of Setup \ref{setup_thm1} in a Ptolemy diagram of a $12$-gon. Green diagonals are dissecting diagonals, with $c=\{v_3,v_6\}$. On the left, $c$ borders an empty cell, so $v_p=v_{3+1}$ and on the right it borders a clique, so $v_q=v_{3-1}$. Then $b_0$ and $b_1$ are the red dashed diagonals.}
    \label{fig:ext_proj}
\end{figure}
\begin{proof}[Proof of Proposition \ref{thm_mras}]
Consider an indecomposable $d\in\mathcal{C}$. By \ref{cluster9}, $\Hom_{\mathcal{T}}(d,c)\neq 0$ if and only if $d$ has one endpoint in each of the intervals $[v_i^{++},v_j]$ and $[v_j^{++}, v_i]$. Since $d\in \mathcal{C}$ and $c$ is a dissecting diagonal, $d$ is not allowed to cross $c$. Hence $d=\{ v_i, t \}$ for $t\in [v_i^{++}, v_j]$ or $d=\{ v_j, s \}$ for $s\in [v_j^{++}, v_i]$. Note that, whenever they are non-zero, our choices of $b_0,\,b_1$ satisfy this condition and so $\dim_{\C}\Hom_{\mathcal{T}}(b_i,c)=1$.

We prove $\beta$ is right almost split in $\mathcal{C}$. Take a morphism $\gamma':c'\rightarrow c$ in $\mathcal{C}$ that is not a split epimorphism. If $\gamma'=0$ it clearly factorizes through $\beta$, so assume $\gamma'$ is non-zero and without loss of generality assume $c'$ is indecomposable. Note that $\gamma'$ being not a split epimorphism forces $c'\neq c$.
By the above, we have either $c'=\{ v_i, t \}$ for $t\in [v_i^{++}, v_{j-1}]$ or $c'=\{ v_j, s \}$ for $s\in [v_j^{++}, v_{i-1}]$.
In the first case, by maximality of $v_p$ in $[v_i^{++},v_j]$ such that $b_0=\{v_i, v_p\}\in\mathcal{C}$, we have $t\leq v_p<v_j$. Then, by \ref{cluster9}, $\gamma'$ factors through $\beta_0$. Similarly, in the second case we have $s\leq v_q<v_i$ and $\gamma'$ factors through $\beta_1$. Hence $\beta$ is right almost split.

We now show that $\beta$ is right minimal. Note that $b_1=\{ v_j,v_q \}$ and $\Sigma^{-1} b_0=\{v_i^+,v_p^+\}$ do not cross. In fact we have
\begin{align*}
v_i<v_i^+<v_i^{++}<v_p^+<v_j<v_j^{++}\leq v_q.
\end{align*}
Similarly,  $b_0$ and $\Sigma^{-1} b_1$ do not cross. Then any morphism
$\varphi:b_0\oplus b_1\rightarrow b_0\oplus b_1$ has the form
\begin{align*}
\varphi=\begin{pmatrix}
\alpha_0 1_{b_0} & 0\\ 0& \alpha_1 1_{b_1}
\end{pmatrix}:b_0\oplus b_1\rightarrow b_0\oplus b_1, 
\end{align*}
where $\alpha_0,\,\alpha_1\in\mathbb{C}$. If $\varphi$ is such that $\beta\circ \varphi=\beta$, then we must have $\alpha_0=\alpha_1=1$, so that $\varphi$ is an isomorphism. Hence $\beta$ is a minimal right almost split morphism in $\mathcal{C}$.

The rest of the proposition follows from \ref{cluster4} and Theorem \ref{prop_mras}. 
\end{proof}

\begin{remark}
In the situation of Proposition \ref{thm_mras}, by Theorem \ref{prop_mras}(b) we also have that if $\beta': b'\rightarrow c'$ is a minimal right almost split morphism in $\mathcal{C}$ with $c'$ $\Ext$-projective, then $c'\cong c$ if and only if $x'\cong x$, where $x'\xrightarrow{\xi'} b'\xrightarrow{\beta'} c'\rightarrow\Sigma x'$ is the triangle obtained by extending $\beta'$.
\end{remark}

We now apply Corollary \ref{coro_mutation} to this example.
\begin{remark}\label{remark_2cy}
Note that, by the dimension of $\Hom$ spaces over $\mathbb{C}$, we have that $\End(z)\cong \mathbb{C}$ for every indecomposable $z$ in $\mathcal{T}$.
Moreover, since $\mathcal{T}$ is $2$-Calabi-Yau, by Remark \ref{remark_2CY_Ext}, we have that $\Ext$-projective and $\Ext$-injective objects coincide in additive subcategories of $\mathcal{T}$.
\end{remark}

\begin{proposition}\label{thm_mutation_eg}
In the situation of Setup \ref{setup_thm1}, consider the triangle $(\dag):x\rightarrow b_0\oplus b_1\rightarrow c\rightarrow \Sigma x$ from Proposition \ref{thm_mras}. Let $\widetilde{\mathcal{C}}\subseteq \mathcal{T}$ be the additive subcategory with $\text{Ind }\widetilde{\mathcal{C}}=\text{Ind }(\mathcal{C})\setminus c$ and $\mathcal{C}':=\text{add }(\widetilde{\mathcal{C}}\cup x)$. Then $\mathcal{C}'\subseteq\mathcal{T}$ is closed under extensions if and only if $c$ borders two empty cells in the Ptolemy diagram corresponding to $\mathcal{C}$.
\end{proposition}

\begin{proof}
Suppose first that $c$ borders two empty cells in $\mathcal{C}$. Then, using the notation in Setup \ref{setup_thm1}, the only diagonals in $\mathcal{C}'$ that $x=\{v_p, v_q\}$ crosses are from the set of diagonals of the form $\{ v_s,v_t \}$ with $s,t\in\{1,\dots,m\}$ that are in $\widetilde{\mathcal{C}}$. However, since $c$ borders two empty cells, such diagonals $\{ v_s,v_t \}$ do not belong to $\widetilde{\mathcal{C}}\subset\mathcal{C}'$. Hence $c$ crosses no diagonals in $\mathcal{C}'$, so that $\Ext^1(x,\mathcal{C}')=0$ and $x$ is $\Ext$-projective in $\mathcal{C}'$, and so also $\Ext$-injective in $\mathcal{C}'$ by Remark \ref{remark_2cy}. By Corollary \ref{coro_mutation}, we then have that $\mathcal{C}'$ is closed under extensions.

Suppose now that one of the cells bordered by $c$ is a clique with at least four vertices. Using the notation in Setup $\ref{setup_thm1}$, without loss of generality say that the cell $v_j<v_{j+1}<\dots<v_{i-1}<v_i$ is a clique with at least four vertices. Then, $x=\{v_p, v_{i-1}\}$ and, since the clique has at least four vertices, we have $v_j<v_{j+1}<v_{i-1}<v_i$. Then the diagonal $\widetilde{c}:=\{ v_{j+1}, v_i \}\in\widetilde{\mathcal{C}}$ crosses $x$ since
\begin{align*}
    v_i<v_{i+1}\leq v_p\leq v_{j-1}<v_j< v_{j+1}< v_{i-1}.
\end{align*}
Then $\Ext^1(x,\widetilde{c})\neq 0$ so that $x$ is not $\Ext$-projective, and so also not $\Ext$-injective, in $\mathcal{C}'$. By Corollary \ref{coro_mutation}, we have that $\mathcal{C}'$ is not closed under extensions.
\end{proof}

\begin{exmp}
In the example illustrated in Figure \ref{fig:ext_proj}, we have that the dissecting diagonal $c$ borders an empty cell and a clique with four vertices. Then, by Proposition \ref{thm_mutation_eg}, the subcategory $\mathcal{C}'$ obtained by removing $c$ and substituting it with $x=\{ v_4,v_2 \}$ is not closed under extensions. In fact, it is easy to see that this does not correspond to a Ptolemy diagram.

However, if the cell to the right of $c$ was empty, then the subcategory obtained by removing $c$ and substituting it with $x=\{ v_1, v_4 \}$ would correspond to a Ptolemy diagram and so it would be closed under extensions.
\end{exmp}

\begin{remark}
Recall that $\mathcal{C}_{A_n}$ is $2$-Calabi-Yau and it has finitely many indecomposables. Hence, whenever $C\in\mathcal{C}$ corresponds to a dissecting diagonal bordering two empty cells, Proposition \ref{thm_mutation_eg} and Theorem \ref{thm_mutZZ_same} imply that $\mathcal{C}'$ is equal to $\mu(\mathcal{C};\mathcal{D})$, \textbf{i.e.} the subcategory obtained by mutating $\mathcal{C}$ with respect to the additive subcategory of $\mathcal{C}$ generated by all the indecomposable $\Ext$-projectives in $\mathcal{C}$ apart from $C$. 
\end{remark}

\subsection{Minimal left almost split morphisms in $\mathcal{C}$ starting at Ext-injectives}
For completeness we state the corresponding results on $\Ext$-injectives. These can be proven using similar arguments to the ones in Section \ref{section35}.

\begin{proposition}\label{thm_inj}
An indecomposable $a$ in $\mathcal{C}$ is $\Ext$-injective if and only if it is a dissecting diagonal.
\end{proposition}

\begin{proof}
This follows from Proposition \ref{thm-extproj} and Remark \ref{remark_2CY_Ext}.
\end{proof}

We present the setup and the dual of Proposition \ref{thm_mras}.

\begin{setup}\label{setup_thm2}
Let $\mathcal{C}$ correspond to a Ptolemy diagram and $a\in \mathcal{C}$ be $\Ext$-injective and indecomposable. Then $a$ is a dissecting diagonal by Proposition \ref{thm_inj}. Let the vertices of the two cells bordered by $a$ be $v_1<v_2< \cdots <v_m$ and $a=\{v_r,v_s\}$. Set $v_0:=v_m$.

Choose $v_p$ minimal in $[v_r^+,v_s^-]$ such that $b_0:=\{ v_s,v_p \}\in \mathcal{C}$ and $v_q$ minimal in $[v_s^+,v_r^-]$ such that $b_1:=\{ v_r, v_q \}\in\mathcal{C}$.

\end{setup}

\begin{proposition}\label{thm_dual_eg}
In the situation of Setup \ref{setup_thm2}, let
\begin{align*}
    \alpha=\begin{psmallmatrix}\alpha_0\\\alpha_1\end{psmallmatrix}: a\rightarrow b_0\oplus b_1,
\end{align*}
where $\alpha_0,\alpha_1$ are non-zero morphisms (unless $b_0$ or $b_1$ are zero). Then $\alpha$ is a minimal left almost split morphism in $\mathcal{C}$ and
\begin{align}
a\xrightarrow{\alpha=\begin{psmallmatrix}\alpha_0\\\alpha_1\end{psmallmatrix}} b_0\oplus b_1\xrightarrow{\nu=(\nu_0,\nu_1)} z\rightarrow \Sigma a \tag{\ddag}
\end{align}
is a triangle in $\mathcal{T}$ with $z=\{ v_p,v_q \}$ indecomposable not in $\mathcal{C}$ and $\nu$ a $\mathcal{C}$-cover of $z$. In other words, $(\ddag)$ is a right-weak Auslander-Reiten triangle in $\mathcal{C}$.
\end{proposition}

\begin{proposition}
In the situation of \ref{setup_thm2}, consider the triangle $a\rightarrow b_0\oplus b_1\rightarrow z\rightarrow \Sigma a$ from Proposition \ref{thm_dual_eg}. Let Let $\underline{\mathcal{C}}$ be the additive category with $\text{Ind }\underline{\mathcal{C}}=\text{Ind }(\mathcal{C})\setminus a$ and $\mathcal{C}'':=\text{add }(\underline{\mathcal{C}}\cup z)$. Then $\mathcal{C}''\subseteq \mathcal{T}$ is closed under extensions if and only if $a$ borders two empty cells in the Ptolemy diagram corresponding to $\mathcal{C}$.
\end{proposition}


\begin{thebibliography}{99}

\bibitem{AF} F.\ W.\ Anderson and K.\ R.\ Fuller, {\it Rings and categories of modules}, Springer-Verlag, New York, 1992, second edition.

\bibitem{A} I.\ Assem, D.\ Simson and A. Skowro\'{n}ski, {\it Elements of the representation theory of associative algebras}, LMSST Vol. 65, Cambridge University Press, Cambridge, 2010.

\bibitem{AR} M.\ Auslander and I.\ Reiten, {\it Representation theory of Artin algebras III}, Comm. Algebra \textbf{3} (1975), 239-294.

\bibitem{ARS} M.\ Auslander, I.\ Reiten and S.\ O.\ Smal\o, ``Representation theory of Artin algebras'', Cambridge Studies in Advanced Mathematics \textbf{36}, Cambridge University Press, New York, 1994.

\bibitem{AS} M.\ Auslander and S.\ O.\ Smal\o, {\it Almost split sequences in subcategories}, J. Algebra \textbf{69} (1981), 426-454.

\bibitem{BMRRT} A. B. Buan, R. Marsh, M. Reineke, I. Reiten and G. Todorov, {\it Tilting theory and cluster combinatorics}. Adv. Math. \textbf{204} (2006), 572–618.

\bibitem{CCS} P.\ Caldero, F.\ Chapoton and R.\ Schiffler, {\it Quivers with relations arising from clusters ($A_n$ case)}, Trans. Amer. Math. Soc. \textbf{358} (2006), 1347-1364.

\bibitem{DH} D.\ Happel, {\it Triangulated categories in the representation theory of finite-dimensional algebras}, London Mathematical Society Lecture Notes Series. Cambridge University Press, Cambridge  \textbf{119} (1988).

\bibitem{H} R.\ Hartshorne, {\it Residues and duality}, Springer-Verlag, Berlin, 1966.

\bibitem{HJR} T.\ Holm, P.\ J\o rgensen, and M.\ Rubey, {\it Ptolemy diagrams and torsion pairs in the cluster category of Dynkin type $A_n$}, J. Algebraic Combin. \textbf{34} (2011), 507-523.

\bibitem{IT} K.\ Igusa and G.\ Todorov, {\it Cluster categories coming from cyclic posets}, Comm. Algebra \textbf{43} (2015), 4367-4402.

\bibitem{IY} O.\ Iyama and Y. Yoshino, {\it Mutation in triangulated categories and rigid Cohen-Macaulay modules},  Invent. Math. \textbf{172} (2008), 117-168.

\bibitem{JP}  P.\ J\o rgensen, {\it Auslander-Reiten triangles in subcategories}, J. K-theory \textbf{3} (2009), 583-601.

\bibitem{JP2}  P.\ J\o rgensen, {\it Auslander-Reiten theory over topological spaces}, Comment. Math. Helv. \textbf{79} (2004), 160-182.

\bibitem{K} B.\ Keller, {\it Cluster algebras, quiver representations and triangulated categories}, Triangulated Categories, London Math. Soc. Lecture Note Ser. \textbf{375}, Cambridge Univ. Press, Cambridge (2010), 76-160.

\bibitem{K2} B.\ Keller, {\it On triangulated orbit categories}, Documenta Math. \textbf{10} (2005), 551-581.

\bibitem{Liu} S.\ Liu, {\it Auslander-Reiten theory in a Krull-Schmidt category}, S\~{a}o Paulo J. Math. Sci. \textbf{4} (2010), 425-472. 

\bibitem{Liu2} S.\ Liu, P. Ng and C. Paquette, {\it Almost split sequences and approximations}, Algebr. Represent. Theory \textbf{16} (2013), 1809-1827.

\bibitem{MK} M.\ Kleiner, {\it Approximations and almost split sequences in homologically finite subcategories}, J. Algebra \textbf{198} (1997), 135-163.

\bibitem{Ng} P.\ Ng, A characterization of torsion theories in the cluster category of Dynkin type $A_\infty$, preprint (2010). \url{math.RT/1005.4364}



\bibitem{RV} I.\ Reiten and M.\ Van den Bergh, {\it Noetherian hereditary abelian categories satisfying Serre duality}, J. Amer. Math. Soc. \textbf{15} (2002), 295-366.

\bibitem{R} C. M. Ringel, {\it The category of modules with good filtrations over a quasi-hereditary algebra has almost split sequences}, Math. Z. \textbf{208} (1991), 209-223.


\bibitem{W} C.\ A.\ Weibel, {\it An Introduction to homological algebra}, Cambridge University Press, Cambridge (1997).

\bibitem{ZZ} W.\ Zhou and B.\ Zhu, Mutation of torsion pairs in triangulated categories and its geometric realization, preprint (2011). \url{math.RT/1105.3521v2}
\end{thebibliography}
\end{document}